\documentclass[aop,MSNbibl,seceqn,dvips]{arximspdf}

\doi{10.1214/12-AOP757} 
\volume{41}
\issue{5}
\pubyear{2013}
\firstpage{3112}
\lastpage{3139}

\makeatletter

\newcommand{\rrvert}{\vert}
\newcommand{\llvert}{\vert}
\newcommand{\eqref}[1]{(\ref{#1})}
\newcommand{\iint}{\int\!\!\!\int}
\newtheorem{thmm}{Theorem}[section]
\newtheorem{lem}[thmm]{Lemma}
\newtheorem{prop}[thmm]{Proposition}
\newtheorem{cor}[thmm]{Corollary}
\newproclaim{defi}[thmm]{Definition}
\newproclaim{rem}[thmm]{Remark}

\newcommand{\R}{\mathbb{R}}
\newcommand{\N}{\mathbb{N}}

\newcommand{\ent}{\operatorname{Ent}}
\newcommand{\Osc}{\operatorname{Osc}}
\makeatother

\begin{document}
\begin{frontmatter}

\title{Characterization of Talagrand's transport-entropy inequalities in metric spaces\thanksref{T1}}
\thankstext{T1}{Supported in part by the ``Agence Nationale de la Recherche'' through the Grants ANR 2011 BS01 007 01 and ANR 10 LABX-58.}
\runtitle{Characterization of transport-entropy inequalities}

\begin{aug}
\author[A]{\fnms{N.} \snm{Gozlan}\ead[label=e1]{nathael.gozlan@univ-mlv.fr}},
\author[B]{\fnms{C.} \snm{Roberto}\corref{}\ead[label=e2]{croberto@math.cnrs.fr}\thanksref{t2}}
\and
\author[A]{\fnms{P.-M.} \snm{Samson}\ead[label=e3]{paul-marie.samson@univ-mlv.fr}}

\thankstext{t2}{Supported in part by the European Research Council
through the ``Advanced Grant'' PTRELSS 228032.}
\runauthor{N. Gozlan, C. Roberto and P.-M. Samson}
\affiliation{Universit\'e Paris Est Marne la Vall\'ee, Universit\'e
Paris Est Marne la Vall\'ee and Universit\'e Paris Ouest Nanterre la D\'
efense, and
Universit\'e~Paris~Est~Marne la Vall\'ee}
\address[A]{N. Gozlan\\
P.-M. Samson\\
Laboratoire d'Analyse et de Math\'ematiques\\
\quad Appliqu\'ees (UMR CNRS 8050)\\
Universit\'e Paris Est Marne la Vall\'ee\\
5 bd Descartes\\
77454 Marne la Vall\'ee Cedex 2\\
France\\
\printead{e1}\\
\phantom{E-mail: }\printead*{e3}} 
\address[B]{C. Roberto\\
Laboratoire d'Analyse et de Math\'ematiques\\
\quad Appliqu\'ees (UMR CNRS 8050)\\
Universit\'e Paris Est Marne la Vall\'ee\\
5 bd Descartes\\
77454 Marne la Vall\'ee Cedex 2\\
France\\
and\\
Universit\'e Paris Ouest Nanterre la D\'efense\\
MODAL'X, EA 3454\\
200 avenue de la R\'epublique 92000 Nanterre\\
France\\
\printead{e2}}
\end{aug}

\received{\smonth{7} \syear{2011}}
\revised{\smonth{3} \syear{2012}}

%
\begin{abstract}
We give a characterization of transport-entropy inequalities in metric
spaces. As an application we deduce that
such inequalities are stable under bounded perturbation
(Holley--Stroock perturbation lemma).
\end{abstract}

%

\end{frontmatter}

\section{Introduction}

In their celebrated paper~\cite{OV00}, Otto and Villani proved that,
in a smooth Riemannian setting, the log-Sobolev inequality
implies the Talagrand transport-entropy inequality ${\mathbf{T}_2}$.
Later, Bobkov, Gentil and Ledoux~\cite{bgl} proposed an alternative
proof of this result. Both approaches are based on semi-group arguments.
More recently, the first named author of this paper gave a new proof,
based on large deviation theory, valid on metric spaces~\cite{gozlan}.

In this paper, on the one hand, we give yet another proof of Otto and
Villani's theorem. This proof does not use any semi-group arguments nor
large deviations, and it requires very few structures on the space. We
are thus able to recover and extend the result of~\cite{gozlan} in a
general metric space framework.

On the other hand, we recently introduced in~\cite{GRS11} a new family
of functional inequalities,
called inf-convolution log-Sobolev inequalities. In a Euclidean
framework, we proved
that these inequalities are equivalent to Talagrand transport-entropy
inequalities ${\mathbf{T}}_\alpha$, associated to cost functions $\alpha$
between linear and quadratic. This led to a new characterization of
${\mathbf{T}_2}$ and other transport-entropy inequalities. The present paper
establishes that this equivalence is true in a general metric space
framework and for general cost functions $\alpha$. As a byproduct, we
prove that the inequalities ${\mathbf{T}}_\alpha$ are stable under bounded
perturbation (Holley--Stroock perturbation lemma).

Our strategy is very general and applies to a very large class of
transport-entropy inequalities.

In order to present our results, we need first to fix some notation.

\subsection{Notation and definitions}

We first introduce the notion of optimal transport cost. Then we give
the definition of the transport-entropy inequality and of
the $(\tau)$-log-Sobolev inequality.

\textit{General assumption}. Throughout this paper, $(X,d)$ will
always be a complete, separable metric space such that closed balls are
compact.

\subsubsection{Optimal transport cost and transport-entropy inequality}
Let $\alpha\dvtx\break  \mathbb{R} \to\mathbb{R}^+$ be a continuous function.
Given two probability measures $\nu$ and $\mu$ on
$X$, the \textit{optimal transport cost} between $\nu$ and $\mu$ (with
respect to the cost function $\alpha$) is defined by
\[
\mathcal{T}_\alpha(\nu,\mu):=\inf_\pi \biggl\{ \iint\alpha
\bigl(d(x,y)\bigr) \,d\pi(x,y) \biggr\},
\]
where the infimum runs over all the probability measures $\pi$ on $X
\times X$ with marginals $\nu$ and $\mu$.
The notion of optimal transport cost is very old (it goes back to Monge
\cite{monge}). It has been intensively studied
and it is used in a wide class of problems running from geometry, PDE
theory, probability and statistics; see~\cite{villani}.
Here we focus on the following transport-entropy inequality.

Throughout this paper, the cost functions $\alpha$ will be assumed to
belong to the class of Young functions.

\begin{defi}[(Young functions\setcounter{footnote}{2}\footnote{Note that, contrary to the
definition of some authors, for us,
a Young function cannot take infinite values.})]
A function $\alpha\dvtx \R\to\R^+$ is a \textit{Young function} if
$\alpha$ is an even, convex, increasing function on $\mathbb{R}^+$
such that $\alpha(0)=0$ and $\alpha'(0)=0.$
\end{defi}

\begin{defi}[(Transport-entropy inequality ${\mathbf{T}_\alpha}$)] \label{deftci}
Let $\alpha$ be a Young function; a probability measure $\mu$ on $X$
is said to satisfy the \textit{transport-entropy inequality} (\ref{tac}), for some $C>0$ if
%
\renewcommand{\theequation}{${\mathbf{T}_\alpha(C)}$}
\begin{equation}\label{tac}
\mathcal{T}_\alpha(\nu,\mu)\leq C H(\nu| \mu)\qquad \forall\nu \in
\mathcal{P}(X),
\end{equation}
where
\[
H(\nu|\mu)= \cases{
\displaystyle\int\log\frac{d\nu}{d\mu} \,d
\nu,& \quad$\mbox{if } \nu\ll\mu,$
\vspace*{2pt}\cr
+\infty,&\quad $\mbox{otherwise},$}
\]
is the relative entropy of $\nu$ with respect to $\mu$, and $
\mathcal{P}(X)$ is the set of all probability measures on $X$.
\end{defi}

\begin{rem}\label{remalphaen0}
It can be shown that if $\alpha\dvtx \R\to\R^+$ is an even convex
function such that $\limsup_{x\to0}\frac{\alpha(x)}{x^2}=+\infty$,
then the only probability measures that satisfy the transport
inequality ${\mathbf{T}_\alpha}$ are Dirac masses; see, for example,
\cite{GL07},
Proposition~2.
This is the reason why, in our definition of Young functions, we impose
that $\alpha'(0)=0$.
\end{rem}
Popular Young functions appearing in the literature, as cost functions
in transport-entropy inequalities, are the functions $\alpha_{p_1,p_2}$, defined by
%
\setcounter{equation}{0}
\begin{equation}
\label{alphap1p2}\qquad \displaystyle\alpha_{p_1,p_2}(x):=\cases{ %
|x|^{p_1}, &\quad $\mbox{if } |x|\leq1,$
\vspace*{2pt}\cr
\displaystyle\frac
{p_1}{p_2}|x|^{p_2} +1-\frac{p_1}{p_2}, &\quad$\mbox{if } |x|>1,$}
\qquad p_1\geq2,\ p_2\geq1
\end{equation}
(the case $p_1<2$ can be discarded according to the remark above). When
$p_1=p_2=p$, we use the notation $\alpha_p$ instead of $\alpha_{p,p}.$

Transport-entropy inequalities imply concentration results as shown by
Marton~\cite{marton}; see also~\cite{BG99,ledoux}, and~\cite{gozlan-leonard} for a full introduction to
this notion.

The transport-entropy inequality related to the quadratic cost $\alpha_2(x)=x^2$ is the most studied in the literature. In this case, the
transport-entropy inequality is often referred to as the Talagrand
transport-entropy inequality and is denoted by ${\mathbf{T}_2}$. Talagrand
\cite{talagrand} proved that, on $(\mathbb{R}^n,| \cdot |_2)$
(where $| \cdot |_2$ stands for the Euclidean norm), the standard
Gaussian measure satisfies ${\mathbf{T}_2}$ with the optimal constant \mbox{$C=2$}.

\subsubsection{Log-Sobolev-type inequalities}
The second inequality of interest for us is the log-Sobolev inequality
and, more generally, modified log-Sobolev inequalities. To define these
inequalities properly, we need to introduce additional notation.

Recall that the Fenchel--Legendre transform $\alpha^*$ of a Young
function $\alpha$ is defined by
\[
\alpha^*(y)=\sup_{x\in\R}\bigl\{xy-\alpha(x)\bigr\}\in\R^+\cup\{\infty\}\qquad
\forall y\in\R.
\]
A function $f\dvtx X\to\R$ is said to be \textit{locally Lipschitz} if for
all $x\in X$, there exists a ball $B$ centered at point $x$ such that
\[
\sup_{y,z \in B,\ y \neq z} \frac{|f(y)-f(z)|}{d(y,z)} < \infty.
\]
When $f$ is locally Lipschitz, we define
\[
\bigl|\nabla^+f\bigr|(x)= \cases{ %
\displaystyle\limsup_{y \to x}
\frac{[f( y) - f(x)]_+}{d(y,x)}, & \quad$\mbox{if } x \mbox{ is not an isolated point},$
\vspace*{2pt}\cr
0, & \quad$\mbox{otherwise},$}
\]
and
\[
\bigl|\nabla^-f\bigr|(x)= \cases{ %
\displaystyle\limsup_{y \to x}
\frac{[f( y) - f(x)]_-}{d(y,x)}, & \quad $\mbox{if } x \mbox{ is not an isolated point},$
\vspace*{2pt}\cr
0, & \quad $\mbox{otherwise},$ }
\]
where $[a]_+=\max(a;0)$ and $[a]_-=\max(-a;0)$. Note that $|\nabla^+
f|(x)$ and $|\nabla^-f|(x)$ are finite for all $x\in X$. When $f$ is a
smooth function on a smooth manifold, $|\nabla^+f|$ and $|\nabla^-f|$
equal the norm of the gradient of $f$.

Finally, if $\mu$ is a probability measure on $X$, recall that the
entropy functional $\ent_\mu( \cdot )$ is defined by
\[
\ent_\mu(g) = \int g \log\frac{g}{\int g \,d\mu} \,d\mu \qquad\forall g>0.
\]

\begin{defi}[(Modified log-Sobolev inequality $\mathbf{LSI}_\alpha^{\pm }$)] \label{defmlsi}
Let $\alpha$ be a Young function; a probability measure $\mu$ on $X$
is said to satisfy the \textit{modified log-Sobolev inequality plus}
(\ref{LSIplusA}) for some $A>0$ if
%
\renewcommand{\theequation}{${\mathbf{LSI}}_\alpha^+(A)$}
\begin{equation}\label{LSIplusA}
\ent_\mu\bigl(e^f\bigr)\leq A\int \alpha^*
\bigl(\bigl|\nabla^+f\bigr|\bigr) e^f \,d\mu
\end{equation}
for all locally Lipschitz bounded functions $f\dvtx X \to\mathbb{R}$.

It verifies the \textit{modified log-Sobolev inequality minus} (\ref{LSIminusA}) for some $A>0$ if
%
\renewcommand{\theequation}{${\mathbf{LSI}}_\alpha^-(A)$}
\begin{equation}\label{LSIminusA}
\ent_\mu\bigl(e^f\bigr)\leq A\int \alpha^*
\bigl(\bigl|\nabla^-f\bigr|\bigr) e^f \,d\mu
\end{equation}
for all locally Lipschitz bounded functions $f\dvtx X \to\mathbb{R}$.
\end{defi}

Again, the quadratic cost $\alpha_2(x)=x^2$ plays a special role since
in this case we recognize the usual log-Sobolev inequality
introduced by Gross~\cite{gross}; see also~\cite{stam}. In this case,
we will use the notation $\mathbf{LSI}^\pm.$

Bobkov and Ledoux~\cite{bobkov-ledoux} introduced first the modified
log-Sobolev inequality with the function $\alpha_{2,1}$, in order to
recover the celebrated result by Talagrand~\cite{talagrand91} on the
concentration phenomenon for products of exponential measures. In
particular these authors proved that, with this special choice of
function, the modified log-Sobolev inequality is actually equivalent to
the Poincar\'e inequality.
After them, Gentil, Guillin and Miclo~\cite{ggm} established that the
probability measure
$d\nu_p(x)=e^{-|x|^p}/Z_p$, $x\in\R$ and $p \in(1,2)$ verifies the
modified log-Sobolev inequality associated with the function $\alpha_{2,p}$.
In a subsequent paper~\cite{ggm2},\vadjust{\goodbreak} they generalized their
results to a large class of measures with tails between exponential and
Gaussian; see also~\cite{barthe-roberto,gozlan07,gentil} and~\cite
{papageorgiou}.

Finally, let us introduce the notion of inf-convolution log-Sobolev
inequality. In a previous work~\cite{GRS11}, we proposed
the following inequality:
%
\setcounter{equation}{1}
\begin{equation}
\label{iclsi} \qquad\ent_\mu\bigl(e^f\bigr)\leq
\frac{1}{1-\lambda C} \int\bigl(f-Q_\alpha^\lambda f\bigr)
e^f \,d\mu\qquad \forall f\dvtx X \to\mathbb{R}, \forall\lambda\in(0,1/C),
\end{equation}
where
\[
Q_\alpha^\lambda f(x)=\inf_{y\in X} \bigl\{f(y)+\lambda
\alpha\bigl(d(x,y)\bigr)\bigr\}\qquad \forall x\in X.
\]
We called it inf-convolution log-Sobolev inequality, and we proved that
it is equivalent---in a Euclidean setting---to the transport-entropy
inequality ${\mathbf{T}_\alpha}(C')$, for Young functions $\alpha$ such
that $\alpha'$ is concave. Also, we get an explicit comparison between
the constants $C$ and $C'$, namely $C \leq C' \leq8C$.
Our proof relies in part on the Hamilton--Jacobi semi-group approach
developed by Bobkov, Gentil and Ledoux~\cite{bgl}.

Inequality \eqref{iclsi} is actually a family of inequalities, with a
constant having a specific form [i.e., $1/(1-\lambda C)$] on the
right-hand side.
In this paper, in order to broaden this notion, we will say $(\tau
)$-log-Sobolev inequality, rather than inf-convolution log-Sobolev
inequality, in the following inequality.

\begin{defi}[{[$(\tau)$-log-Sobolev inequality]}] \label{defqlsi}
Let $\alpha$ be a Young function; a probability measure $\mu$ on $X$
is said to satisfy the \textit{$(\tau)$-log-Sobolev inequality} (\ref{tauminusLSI}) for some $\lambda, A>0$
if
%
{\renewcommand{\theequation}{$(\tau)-\mathbf{LSI}_\alpha(\lambda,A)$}
\begin{equation}\label{tauminusLSI}
\ent_\mu\bigl(e^f\bigr)\leq A \int
\bigl(f-Q_\alpha^\lambda f\bigr) e^f \,d\mu
\end{equation}}
\hspace*{-2pt}for all bounded locally Lipschitz functions $f\dvtx X\to\R,$ where the
inf-convolution operator $Q_\alpha^\lambda$ is defined by
%
\setcounter{equation}{2}
\begin{equation}
\label{inf-convoperator} Q_\alpha^\lambda f(x)=
\inf_{y\in X} \bigl\{f(y)+\lambda\alpha\bigl(d(x,y)\bigr)\bigr\}\qquad \forall x
\in X .
\end{equation}
When $\lambda=1$, we use the notation $Q_\alpha$ instead of $Q_\alpha^1.$
\end{defi}
The notation $(\tau)-\mathbf{LSI}_\alpha$ refers to the celebrated
$(\tau)$-Property introduced
by Maurey~\cite{maurey} (that uses the inf-convolution operator
$Q_\alpha$ and that is also closely related to the transport-entropy
inequality; see~\cite{gozlan-leonard}, Section 8.1).

Of course \eqref{iclsi} implies $(\tau)-\mathbf{LSI}_\alpha(\lambda
,1/(1-\lambda C))$, for any $\lambda\in(0,1/C)$.
The other direction is not clear, a priori (it would trivially be true
if $A=1$), even if the two inequalities have the same flavor.
Thanks to Theorem~\ref{GRS-thm-intro} below, they appear to be
equivalent, under mild assumptions on $\alpha$.

\subsubsection{\texorpdfstring{$\Delta_2$-condition}{Delta2-condition}} \label{secyoung}
In the next sections, our objective will be to relate the log-Sobolev
inequalities $\mathbf{LSI}_\alpha$ and $(\tau) - \mathbf{LSI}_\alpha$ to
the transport-entropy inequality~$\mathbf{T}_\alpha$. This program works\vadjust{\goodbreak}
well if we suppose that $\alpha$ verifies the classical doubling
condition $\Delta_2.$ Recall that a Young function $\alpha$ is said
to satisfy the \textit{$\Delta_2$-condition} if there exists some
positive constant $K$ (that must be greater than or equal to $2$)
such that
\[
\alpha(2x) \leq K \alpha(x)\qquad \forall x\in\R.
\]
The classical functions $\alpha_{p_1,p_2}$ introduced in \eqref{alphap1p2} enjoy this condition.

The following observation will be very useful in the sequel.
%
\begin{lem}\label{lemrp}
If $\alpha$ is a Young function satisfying the $\Delta_2$-condition, then
%
\begin{equation}
\label{defrp} r_\alpha:=\inf_{x>0}\frac{x\alpha'_-(x)}{\alpha(x)}\geq1
\quad\mbox{and}\quad 1<p_\alpha:=\sup_{x>0}\frac{x\alpha'_+(x)}{\alpha
(x)}<+\infty,
\end{equation}
where $\alpha'_+$ (resp., $\alpha'_-$) denotes the right (resp., left)
derivative of $\alpha.$
\end{lem}
The proof of this lemma is in the \hyperref[app]{Appendix}.
To understand these exponents $r_\alpha$ and $p_\alpha$, observe that
for the function $\alpha=\alpha_{p_1,p_2}$, defined by \eqref{alphap1p2},
we have $r_\alpha=\min(p_1,p_2)$ and $p_\alpha=\max(p_1,p_2)$.
Moreover, if $1\leq r\leq p$ are given numbers, and $\alpha$ is a
Young function such that $r_\alpha=r$ and $p_\alpha=p$, then it is
not difficult to check that
\[
\alpha(1)\alpha_{p,r}\leq\alpha\leq\alpha(1)\alpha_{r,p}.
\]

\subsection{Main results}

Our first result states that the modified log-Sobolev inequality (plus
or minus) implies the transport-entropy inequality
associated with the same $\alpha$ (Otto--Villani theorem).

\begin{thmm} \label{OV-thm-intro}
Let $\mu$ be a probability measure on $X$ and $\alpha$ a Young
function satisfying the $\Delta_2$-condition.
\begin{longlist}[(ii)]
\item[(i)] If $\mu$ satisfies (\ref{LSIplusA}) for some
$A>0$, then $\mu$ satisfies ${\mathbf{T}_\alpha}(C^+)$ with
\[
C^+=\max \bigl( \bigl((p_\alpha-1)A \bigr)^{r_\alpha-1} ;
\bigl((p_\alpha-1)A \bigr)^{p_\alpha-1} \bigr).
\]
\item[(ii)] If $\mu$ satisfies (\ref{LSIminusA}) for some
$A>0$, then $\mu$ satisfies ${\mathbf{T}_\alpha}(C^-)$ with
\[
C^-= \bigl(1+(p_\alpha-1)A \bigr)^{p_\alpha-r_\alpha}\bigl((p_\alpha
-1)A\bigr)^{r_\alpha-1}.
\]
The numbers $1\leq r_\alpha\leq p_\alpha$, $p_\alpha>1$ are defined
by \eqref{defrp}.
\end{longlist}
\end{thmm}

Let us comment on this theorem. First observe that $C^+$ and $C^-$ are
of the same order since
\[
C^+\leq C^-\leq 2^{p_\alpha-r_\alpha} C^+.
\]

For the quadratic case $\alpha_2(x)=x^2$, the constants reduce to $C^+=C^-=A$.
This corresponds (when $X$ is a smooth Riemannian manifold)
to the usual Otto--Villani theorem~\cite{OV00}; see also~\cite{bgl}.
Let us mention that Lott and Villani~\cite{LV07} generalized\vadjust{\goodbreak}
the result from Riemannian manifolds to length spaces, for $\alpha_2(x)=x^2$,
with an adaptation of the Hamilton--Jacobi semigroup approach developed
by Bobkov, Gentil and Ledoux~\cite{bgl}.
But their statement requires additional assumptions, such as a local
Poincar\'e inequality, which are not needed in Theorem~\ref{OV-thm-intro}.

Also, in~\cite{ggm} the authors prove that the modified log-Sobolev
inequality, in Euclidean setting and with $\alpha=\alpha_{2,p}$, with
$1\leq p \leq2$, implies the corresponding transport inequality ${\mathbf{T}_\alpha}$,
again using the Hamilton--Jacobi approach~\cite{bgl}.

More recently, in~\cite{gozlan}, the first named author proved that
$\mathbf{LSI}^+(A)$ implies ${\mathbf{T}_2}(A)$ in the quadratic case $\alpha_2(x)=x^2$
and on an arbitrary complete and separable metric space. His
proof can be easily extended to more general functions such as $\alpha_p(x)=x^p$.
The scheme of proof is the following. Talagrand's
inequality ${\mathbf{T}_2}$ is first shown to be equivalent to
dimension-free Gaussian concentration. According to the well-known
Herbst argument (see, e.g.,~\cite{ledoux}), $\mathbf{LSI}^+$ implies
dimension-free Gaussian concentration, so it also implies ${\mathbf{T}_2}.$

Finally, as shown by Cattiaux and Guillin~\cite{cattiaux-guillin}, we
mention that the Talagrand transport-entropy inequality ${\mathbf{T}_2}$
does not imply, in general, the log-Sobolev inequality. Hence, there is
no hope to get an equivalence in the above theorem.

However, the $(\tau)$-log-Sobolev inequality appears to be equivalent
to the transport-entropy inequality.
This is the main result of this paper.

\begin{thmm} \label{GRS-thm-intro}
Let $\mu$ be a probability measure on $X$, and $\alpha$ a Young
function satisfying the $\Delta_2$-condition, and let $p_\alpha>1$ be
defined by \eqref{defrp}.
The following statements are equivalent:
\begin{longlist}[(1)]
\item[(1)] there exists $C$ such that $\mu$ satisfies ${\mathbf{T}_\alpha}(C)$;
\item[(2)] there exist $\lambda$, $A >0$ such that $\mu$ satisfies
(\ref{tauminusLSI}). 
\end{longlist}
Moreover, the constants are related in the following way:
\begin{eqnarray*}
&&(1)\quad\Rightarrow\quad(2)\qquad \mbox{for any } \lambda\in(0,1/C) \mbox{ and } A=
\frac{1}{1-\lambda C};
\\
&&(2)\quad\Rightarrow\quad(1)\qquad\mbox{with } C= \frac{1}{\lambda} \kappa_{p_\alpha}
\max(A;1)^{p_\alpha-1},
\end{eqnarray*}
where $\kappa_{p_\alpha}=\frac{p_\alpha^{p_\alpha(p_\alpha
-1)}}{(p_\alpha-1)^{(p_\alpha-1)^2}}$.
\end{thmm}

Such a characterization appeared for the first time in~\cite{GRS11},
in a Euclidean setting and with $\alpha$ between linear and quadratic.
Here our result is valid, not only for a wider family of Young
functions $\alpha$, but also on very general metric spaces.

Due to its functional form, it is easy to prove a perturbation lemma
for the inequality $(\tau)-\mathbf{LSI}_\alpha$.
This leads to the following general Holley--Stroock perturbation result
for transport-entropy inequalities whose proof is given in Section~\ref{hs}.\vadjust{\goodbreak}

\begin{thmm} \label{holley}
Let $\mu$ be a probability measure on $X$ and $\alpha$ a Young
function satisfying the $\Delta_2$-condition, and let $p_\alpha>1$ be
defined by \eqref{defrp}.
Assume that $\mu$ satisfies ${\mathbf{T}_\alpha}(C)$ for some constant
$C>0$. Then, for any bounded function $\varphi\dvtx X \to\mathbb{R}$,
the measure $d\tilde\mu= \frac{1}{Z} e^\varphi \,d\mu$ (where $Z$
is the normalization constant)
satisfies ${\mathbf{T}_\alpha} (\widetilde{C} )$, with
\[
\widetilde{C}=\widetilde{\kappa}_{p_\alpha} C e^{(p_\alpha
-1)\mathrm{Osc}(\varphi)},
\]
where $\Osc(\varphi):=\sup\varphi- \inf\varphi$, and $\widetilde
{\kappa}_{p_\alpha}=\frac{p_\alpha^{p_\alpha^2}}{(p_\alpha
-1)^{p_\alpha(p_\alpha-1)}}$.
\end{thmm}

This theorem fully extends the previous perturbation result~\cite{GRS11}, Corollary~1.8, obtained in a Euclidean setting and for a Young
function $\alpha$ such that $\alpha'$ is concave. Namely, for such an
$\alpha$, the function $\alpha(x)/x^2$ is nonincreasing~\cite{GRS11}, Lemma 5.6, and so $p_\alpha\leq2$.

The paper is divided into five sections and one \hyperref[app]{Appendix}. Section \ref
{preliminaries} is dedicated to some preliminaries. In particular we
will give
a characterization of transport-entropy inequalities (close from Bobkov
and G\"otze one) that might be of independent interest, and that is one
of the main ingredients in our proofs.
For the sake of completeness, we also recall how the transport-entropy
inequality ${\mathbf{T}_\alpha}$ implies the $(\tau)$-log-Sololev
inequality $(1) \Rightarrow(2)$ of Theorem
\ref{GRS-thm-intro}; this argument had been first used in~\cite{S00}
and then in~\cite{GRS11}. In Section~\ref{GRS-section},
we prove the other direction: the $(\tau)$-log-Sololev inequality
implies the transport-entropy inequality ${\mathbf{T}_\alpha}$. In Section~\ref{OV-section},
we give the proof of the generalized Otto--Villani
result, Theorem
\ref{OV-thm-intro}. The proof of the Holley--Stroock perturbation
result is given in Section~\ref{hs}. Finally, most of the technical
results needed on Young functions are proved in the \hyperref[app]{Appendix}.

\section{Preliminaries}\label{preliminaries}

In this section, we first recall the proof of the first half of Theorem
\ref{GRS-thm-intro}, namely ${\mathbf{T}_\alpha}\Rightarrow(\tau
)-\mathbf{LSI}_\alpha$. In a second part, we give a useful
``dimensional'' refinement of the characterization of transport-entropy
inequalities by Bobkov and G\"otze~\cite{BG99}. This characterization
provides sufficient conditions for the transport-entropy inequality to
hold. These are the same conditions as those obtained in the proofs of
$\mathbf{LSI}_\alpha^\pm\Rightarrow{\mathbf{T}_\alpha}$ and
$(\tau)-\mathbf{LSI}_\alpha\Rightarrow{\mathbf{T}_\alpha}$.

\subsection{\texorpdfstring{From transport entropy to $(\tau)$-log-Sobolev inequality}
{From transport entropy to (tau)-log-Sobolev inequality}}
In~\cite{GRS11}, Theorem~2.1, we proved
the following result which is the first half [$(1)$ $\Rightarrow$
$(2)$] of Theorem~\ref{GRS-thm-intro}. For the sake of completeness,
its short proof is recalled below.

\begin{thmm}[({\cite{GRS11}})] \label{grs}
Let $\mu$ be a probability measure on $X$ and $\alpha$ a Young function.
If $\mu$ satisfies ${\mathbf{T}_\alpha}(C)$ for some constant $C>0$, then,
for all $\lambda\in(0,1/C)$,
$\mu$ satisfies $(\tau)-\mathbf{LSI}_\alpha(\lambda, \frac
{1}{1-\lambda C})$.\vadjust{\goodbreak}
\end{thmm}
\begin{pf}
Take $f\dvtx X\to\R$ a locally Lipschitz function such that $\int e^f
\,d\mu=1$, and consider the probability $\nu_f$ defined by $\nu_f=e^f
\mu.$ Jensen's inequality implies that $\int f \,d\mu\leq0$. So, if
$\pi$ is an optimal coupling between $\nu_f(dx)$ and $\mu(dy)$, then
it holds
\[
H(\nu_f|\mu)=\int fd\nu_f\leq\int fd\nu_f-
\int f \,d\mu=\int f(x)-f(y) \pi(dx\,dy).
\]
By definition of $Q_\alpha^\lambda f$,
\[
f(x)-f(y)\leq f(x)-Q_\alpha^\lambda f(x)+\lambda\alpha
\bigl(d(x,y)\bigr).
\]
Since $\pi$ is optimal, it holds
\[
H(\nu_f|\mu)\leq\int\bigl(f-Q_\alpha^\lambda f\bigr)
\,d\nu_f + \lambda \mathcal{T}_\alpha(\nu_f,
\mu).
\]
Plugging the inequality $\mathcal{T}_\alpha(\nu_f,\mu)\leq CH(\nu_f|\mu)$
into the inequality above with $\lambda<1/C$ immediately
gives $(\tau)-\mathbf{LSI}_\alpha(\lambda, \frac{1}{1-\lambda C})$.
\end{pf}


\subsection{Sufficient conditions for transport-entropy inequality}

In this section, we show that bounds on the exponential moment of the
tensorized inf-convolution or sup-convolution operator allow us to
recover the transport-entropy inequality; see Proposition~\ref{p11}
and Corollary~\ref{cor} below. These results are a key argument to
recover the transport-entropy inequality, either from a modified
log-Sobolev or from a $(\tau)$-log-Sobolev inequality.

It is known, since the work by Bobkov and G\"otze~\cite{BG99} (see
also~\cite{villani,gozlan-leonard}), that transport-entropy
inequalities have the following dual formulation.

\begin{prop}[(\cite{BG99})] \label{bg}
Let $\mu$ be a probability measure on a complete and separable metric
space $(X,d)$. Then the following are equivalent:
\begin{longlist}[(ii)]
\item[(i)] the probability measure $\mu$ satisfies ${\mathbf{T}_\alpha}(1/c)$;
\item[(ii)] for any bounded continuous function $f \dvtx  X \to\mathbb
{R}$, it holds
\[
\int e^{c Q_\alpha f} \,d\mu\leq e^{c \mu(f)} .
\]
\end{longlist}
\end{prop}

In the next proposition we show, using the law of large numbers, that
the bound in point (ii) can be relaxed as soon as it holds in any dimension.

\begin{prop}\label{p11}
Let $\mu$ be a probability measure on a complete and separable metric
space $(X,d)$.
Then the following are equivalent:
\begin{longlist}[(ii)]
\item[(i)] the probability measure $\mu$ satisfies ${\mathbf{T}_\alpha}(1/c)$;
\item[(ii)] there exist three constants $a$, $b$, $c>0$ such that
for any $n \in\mathbb{N}^*$, for any bounded continuous function $f \dvtx
X^n \to\mathbb{R}^+$, it holds
\[
\int e^{c Q_{\alpha,n} f} \,d\mu^{n} \leq a e^{b \mu^{n}(f)},\vadjust{\goodbreak}
\]
where
%
\begin{equation}
\label{Qdimn}\qquad Q_{\alpha,n} f(x)= \inf_{y \in X^n} \Biggl\{f(y) +\sum
_{i=1}^n \alpha\bigl(d(x_i,y_i)
\bigr) \Biggr\}\qquad \forall x=(x_1,\ldots,x_n)\in
X^n.
\end{equation}
\end{longlist}
\end{prop}

\begin{rem}
Note that the constants $a$ and $b$ do not play any role. On the other
hand, notice that $f$ is only assumed to be nonnegative.
\end{rem}

\begin{pf}
Observe that the transport-entropy inequality ${\mathbf{T}_\alpha}(1/c)$
naturally tensorises; see, e.g.,~\cite{gozlan-leonard}. Applying
Bobkov and G\"otze result above, we see that (i) implies (ii) with
$a=1$ and $b=c$.

Now let us prove that (ii) implies (i). For that purpose, fix a
bounded continuous function $f \dvtx  X \to\mathbb{R}$ with mean $0$ under
$\mu$ and, following~\cite{gozlan-roberto-samson-JFA} (see also \cite
{gozlan}), define~$g$ on $X^n$ as
$g(x)= \sum_{i=1}^n f(x_i)$, $x=(x_1,\dots,x_n) \in X^n$. Then,
\[
\biggl( \int e^{c Q_\alpha f} \,d\mu \biggr)^n = \int
e^{cQ_{\alpha,n} g} \,d\mu^{n} \leq \int e^{cQ_{\alpha, n} g_+} \,d
\mu^{n} \leq a e^{b \mu^{n}(g_+)},
\]
where, as usual, $g_+ = \max(g,0)$. It follows that
\[
\int e^{c Q_\alpha f} \,d\mu\leq a^{{1}/{n}} e^{b \mu^{n}(g_+)/n} .
\]
Now, according to the strong law of large numbers,
$\frac{1}{n} \sum_{i=1}^n f(X_i) \to0$ in $\mathbb{L}^1$, where the
$X_i$'s are i.i.d. random variables with common law $\mu$. Hence,
\[
\mu^{n} \biggl( \frac{g_+}{n} \biggr) \leq\mathbb{E} \Biggl(
\Biggl\llvert \frac{1}{n} \sum_{i=1}^n
f(X_i) \Biggr\rrvert \Biggr) \to0
\]
when $n$ tends to infinity. We conclude that
%
\begin{equation}
\label{lipschitz} \int e^{c Q_\alpha f} \,d\mu\leq1 = e^{c \mu(f)} .
\end{equation}
Since the latter is invariant by changing $f$ into $f+e$ for any
constant $e$, we can remove the assumption $\mu(f)=0$. This ends the proof.
\end{pf}

The next corollary will be used in the proofs of Theorems \ref
{OV-thm-intro} and~\ref{GRS-thm-intro}. It gives a sufficient
condition for
the transport-entropy inequality ${\mathbf{T}_\alpha}$ to hold.

\begin{cor} \label{cor}
Let $\mu$ be a probability measure on a complete and separable metric
space $(X,d)$.
Define, for all $f\dvtx X^n\to\R$,
%
\begin{equation}
\label{Pdimn}  P_{\alpha,n} f(x)= \sup_{y \in X^n} \Biggl\{f(y) - \sum
_{i=1}^n \alpha\bigl(d(x_i,y_i)
\bigr) \Biggr\} \qquad\forall x=(x_1,\ldots,x_n)\in
X^n.\hspace*{-35pt}
\end{equation}
Assume that there exist some constants $\tau$, $a$, $b>0$ and $c \in
[0,1)$ such that, for all integer $n \in\mathbb{N}^*$ and
all bounded continuous functions $f \dvtx  X^n \to\mathbb{R}^+$, it holds
\[
\int e^{\tau P_{\alpha,n}f} \,d\mu^n \leq ae^{b \mu^n(P_{\alpha,n}f)}
e^{\tau c\|f\|_\infty} .
\]
Then $\mu$ satisfies ${\mathbf{T}_\alpha} (\frac1{\tau(1-c)} )$.
\end{cor}

\begin{pf}
Let $n\in\N^*$ and take a bounded continuous function $g \dvtx  X^n \to
\mathbb{R}^+$.
In order to apply Proposition~\ref{p11}, we need to remove the
spurious term $\|f\|_\infty$.
Observe on the one hand that for any $\beta\in(0,\tau(1-c))$, one has
\begin{eqnarray*}
\int e^{\beta Q_{\alpha,n} g} \,d\mu^n&=& 1 + \beta\int_0^{+\infty}
e^{\beta r} \mu^n( Q_{\alpha,n}g \geq r) \,dr
\\
&= & 1 + \beta\int_0^{+\infty} e^{\beta r}
\mu^n\bigl( \min(Q_{\alpha
,n}g, r) \geq r\bigr) \,dr .
\end{eqnarray*}
On the other hand, set $f= \min(Q_{\alpha,n}g, r)$. It is bounded,
nonnegative and satisfies $\|f\|_\infty\leq r$.
Moreover, we have $P_{\alpha,n}(Q_{\alpha,n} g)\leq g$. Indeed,
\[
P_{\alpha,n}(Q_{\alpha,n}g) (x)=\sup_{y\in X^n}
\inf_{z\in X^n} \Biggl\{f(z)+\sum_{i=1}^n
\alpha\bigl(d(y_i,z_i)\bigr)-\sum
_{i=1}^n\alpha \bigl(d(x_i,y_i)
\bigr) \Biggr\},
\]
and the inequality follows by taking $z=x$.
Hence
$\mu^n(P_{\alpha,n}f) =\break  \mu^n(P_{\alpha,n}(Q_{\alpha,n} g))\leq
\mu^n(g)$.
Therefore, since $P_{\alpha,n}f \geq f$, by Chebyshev's inequality and
the assumption, we have
\begin{eqnarray*}
\mu^n\bigl(\min(Q_{\alpha,n}g, r)\geq r\bigr) & \leq&
\mu^n(P_{\alpha,n} f \geq r) \leq e^{-\tau r}\int
e^{\tau P_{\alpha,n}f} \,d\mu^n\\
& \leq& a e^{b \mu^n(P_{\alpha,n}f)} e^{-\tau(1-c)r}
\\
& \leq& a e^{b \mu^n(g)} e^{-\tau(1-c)r}.
\end{eqnarray*}
Consequently, we get
\begin{eqnarray*}
\int e^{\beta Q_{\alpha,n} g} \,d\mu^n & \leq& 1+\beta a e^{b \mu^n(g)}
\int_0^{+\infty} e^{-(\tau(1-c)-\beta)r} \,dr \\
&=& 1 +
\frac{\beta a}{\tau(1-c)-\beta} e^{b \mu^n(g)}
\\
& \leq& \frac{\tau(1-c) + \beta(a-1)}{\tau(1-c)-\beta} e^{b \mu^n(g)} .
\end{eqnarray*}
Finally, Proposition~\ref{p11} provides that $\mu$ satisfies
${\mathbf{T}_\alpha}(1/\beta)$. Optimizing over $\beta$ leads to the expected result.
\end{pf}

\section{From modified log-Sobolev inequality to transport-entropy
inequality} \label{OV-section}

In this section we prove Theorem~\ref{OV-thm-intro}.
We have to distinguish between the modified log-Sobolev inequalities
plus and minus. As in~\cite{gozlan}, the proofs of Theorems \ref
{OV-thm-intro} and~\ref{GRS-thm-intro} use as a main ingredient the
stability of log-Sobolev-type inequalities under tensor products.

Let us recall this tensorisation property. The entropy functional
enjoys the following well-known sub-additivity property (see, e.g.,
\cite{ane}, Chapter 1): if $h\dvtx X^n\to\R^+$,
%
\begin{equation}
\label{subadditivity} \ent_{\mu^n}(h)\leq\sum
_{i=1}^n \int\ent_{\mu}
(h_{i,x}) \,d\mu^n(x),
\end{equation}
where, for all $x\in X^n$, the application $h_{i,x}$ is the $i$th
partial application defined by
\[
h_{i,x}(u)=h(x_{1},\ldots,x_{i-1},u,x_{i+1},
\ldots,x_{n})\qquad \forall u\in X.
\]
Let us say that $h\dvtx X^n\to\R^+$ is \textit{separately locally
Lipschitz}, if all the partial applications $h_{i,x}$ $1\leq i\leq n$,
$x\in X^n$ are locally Lipschitz on $X.$
Now, suppose that a probability $\mu$ on $X$ verifies
(\ref{LSIplusA}) for some $A>0.$ Then, using \eqref{subadditivity},
we easily conclude that $\mu^n$ enjoys the following inequality:
%
\begin{equation}
\label{mlsi+} \ent_{\mu^n}\bigl(e^f\bigr)\leq A\int\sum
_{i=1}^n \alpha^*\bigl(\bigl|
\nabla_i^+ f\bigr|\bigr) e^f \,d\mu^n 
\end{equation}
for all functions $f\dvtx X^n\to\R$ separately locally Lipschitz, where
$|\nabla_{i}^+f|(x)$ is defined by
\[
\bigl|\nabla_{i}^+f\bigr|(x)=\bigl|\nabla^+ f_{i,x}\bigr|(x_{i})=
\limsup_{y\to x_{i}} \frac{[f(x_{1},\ldots,x_{i-1},y,x_{i+1},\ldots,x_{n})-f(x)]_{+}}{d(y,x_{i})}.
\]
The same property holds for $\mathbf{LSI}^-_\alpha$.

\subsection{Modified log-Sobolev inequality plus}
The first part of Theorem~\ref{OV-thm-intro}, that we restate below,
says that the modified log-Sobolev inequality
$\mathbf{LSI}^+_\alpha$ implies
the transport-entropy inequality ${\mathbf{T}_\alpha}$. In fact we shall
prove the following, slightly stronger, result. To any Young function
$\alpha$, we associate a function $\xi_\alpha$ defined by
%
\begin{equation}
\label{xi} \xi_\alpha(x):= \sup_{u >0} \frac{\alpha^*(x\alpha_+'(u))}{x\alpha(u)} ,\qquad x >
0,
\end{equation}
where $\alpha'_+$ is the right derivative of $\alpha$.
Note that $\xi_\alpha$ is nondecreasing and may take infinite values.

\begin{thmm}\label{OV-thm}
Let $\mu$ be a probability measure on $X$ and $\alpha$ a Young
function satisfying the $\Delta_2$-condition.
If $\mu$ satisfies (\ref{LSIplusA}) for some constant $A>0$,
then~$\mu$ satisfies ${\mathbf{T}_\alpha}(1/t_A)$ with $t_A=\sup\{t \in
\R^+ ; \xi_\alpha(t) < 1/A \}$.
\end{thmm}

The following lemma gives an estimation of $\xi_\alpha$.
%
\begin{lem}\label{lemestimationxi}
Let $\alpha$ be a Young function satisfying the $\Delta_2$-condition,
and let $1\leq r_\alpha\leq p_\alpha$, $p_\alpha>1$ be the numbers
defined by \eqref{defrp}. Then, it holds
%
\begin{equation}
\label{eqestimationxi} \xi_\alpha(x)\leq(p_\alpha-1)\max
\bigl(x^{{1}/{(p_\alpha-1)}}; x^{{1}/{(r_\alpha-1)}} \bigr) \qquad\forall x>0,
\end{equation}
with the convention $x^\infty=0$ if $x\leq1$ and $\infty$ otherwise.
\end{lem}
The proof of this result is in the \hyperref[app]{Appendix}.

Using Lemma~\ref{lemestimationxi}, we easily derive point (i) of Theorem \ref
{OV-thm-intro}, with the explicit constant $C^+=\max (
((p_\alpha-1)A )^{r_\alpha-1} ; ((p_\alpha-1)A
)^{p_\alpha-1} ).$

Before turning to the proof of Theorem~\ref{OV-thm}, let us say that
estimation \eqref{eqestimationxi} is satisfactory, at least for the
small values of $x$ (corresponding to the large values of~$A$), as we
show with the following exact calculation of $\xi_\alpha$, when
$\alpha$ is the function $\alpha_{p_1,p_2}$ defined by \eqref{alphap1p2}.
%
\begin{lem}\label{exactcalculation}
Let $p_1\geq2$ and $p_2>1$, and let $\alpha=\alpha_{p_1,p_2}$; then
$p_\alpha=\break\max(p_1,p_2)$, and it holds
\[
\xi_\alpha(x)=(p_\alpha-1)x^{{1}/{(p_\alpha-1)}} \qquad\forall x\leq1.
\]
Moreover, for $x\geq1$, it holds
\[
\xi_\alpha(x)=\cases{ %
\displaystyle p_1
\biggl(\frac{1}{q_2}x^{{1}/{(p_2-1)}}+ \biggl(\frac{1}{q_1}-
\frac{1}{q_2} \biggr)\frac{1}{x} \biggr), &\quad  $\mbox{if } p_1
\geq p_2,$
\vspace*{2pt}\cr
\displaystyle\max \bigl((p_1-1)x^{{1}/{(p_1-1)}};(p_2-1)x^{
{1}/{(p_2-1)}}
\bigr), & \quad$\mbox{if } p_1\leq p_2 ,$ }
\]
where $q_1=p_1/(p_1-1)$ and $q_2=p_2/(p_2-1).$
\end{lem}
The proof of this lemma is in the \hyperref[app]{Appendix}, too.

\begin{pf*}{Proof of Theorem~\ref{OV-thm}}
Our aim is to use Herbst's argument (see, e.g., \cite
{ledoux,helffer,ane}) together with Proposition~\ref{p11}.
Let $n\in\N^*$; according to Lemma~\ref{locallyLip} below, for any
bounded function $f\dvtx X^n\to\R$, the function $Q_{\alpha,n}f$ is
separately locally Lipschitz [recall that the inf-convolution operator
$Q_{\alpha,n}$ is defined by \eqref{Qdimn}].
Fix a nonnegative bounded continuous function $f\dvtx  X^n \to\mathbb
{R}^+$. Applying \eqref{mlsi+} to $tQ_{\alpha,n}f$, $t>0$, and using
Lemma~\ref{lemgradient+} below together with the fact that $f\geq0$,
one gets
\begin{eqnarray*}
\ent_{\mu^n} \bigl(e^{tQ_{\alpha,n}f} \bigr) &\leq& A\int\sum
_{i=1}^n \alpha^* \bigl(t \bigl|\nabla_i^+
Q_{\alpha,n} f\bigr| \bigr) e^{tQ_{\alpha,n}f} \,d\mu^n
\\
&\leq& A t \xi_\alpha(t) \int Q_{\alpha,n}f e^{tQ_{\alpha,n}f} \,d
\mu^n.
\end{eqnarray*}
Now, we proceed with the Herbst argument. Set $H(t)=\int e^{tQ_{\alpha
,n}f}\,d\mu^n$, $t>0$. Since
$\ent_{\mu^n} (e^{tQ_{\alpha,n}f} ) = tH'(t) - H(t) \log
H(t)$, the latter can be rewritten as
\[
\bigl(t-At\xi_\alpha(t)\bigr) H'(t)\leq H(t) \log H(t), \qquad t
>0.
\]
Set $W(t)=\frac{1}{t} \log(H(t))$, $t>0$, so that the previous
differential inequality reduces to
\[
W'(t) t\bigl(1-A\xi_\alpha(t)\bigr) \leq A
\xi_\alpha(t) W (t).
\]
Since $\lim_{t \to0} W(t)= \mu^n(Q_{\alpha,n}f)$, we get
\[
H(t) \leq\exp \bigl( t C(t) \mu^n(Q_{\alpha,n}f) \bigr)\qquad \forall
t\in(0,t_A ),
\]
where we set $C(t)=\exp\int_{0}^t \frac{A\xi_\alpha(u)}{u(1-A\xi_\alpha(u))} \,du$; thanks to Lemma~\ref{lemestimationxi} above, we
are guaranteed that $t_A>0$ and that $C(t)<\infty$ on $(0,t_{\alpha})$.
Since $Q_{\alpha,n}f \leq f$, we finally get
\[
\int e^{tQ_{\alpha,n}f}\,d\mu^n \leq e^{tC(t) \mu^n(f)}\qquad \forall t
\in(0,t_A ),
\]
which leads to the expected result, thanks to Proposition~\ref{p11}
[and after optimization over $t \in(0,t_A )$].
\end{pf*}

\begin{lem}\label{locallyLip}Let $\alpha$ be a Young function. For
any integer $n\in\N^*$, any bounded function $f\dvtx X^n\to\R$, the
function $Q_{\alpha,n}f$ is separately locally Lipschitz on $X^n.$
\end{lem}
\begin{pf}
Let $h=Q_{\alpha,n}f$; then, for all $x\in X^n$ and $1\leq i\leq n$,
it holds
\begin{eqnarray*}
h_{i,x}(u)&=&\inf_{y_{i}\in X} \biggl\{ \inf_{y_{1},\ldots
,y_{i-1},y_{i+1},\ldots,y_{n}} \biggl
\{f(y)+\sum_{j\neq i} \alpha \bigl(d(x_{j},y_{j})
\bigr) \biggr\}+\alpha\bigl(d(u,y_{i})\bigr) \biggr\}
\\
& =& Q_{\alpha} g(u),
\end{eqnarray*}
where $g\dvtx X\to\R$ is defined by the second infimum. Let us show that
$u\mapsto Q_{\alpha} g(u)$ is locally Lipschitz on $X.$
Observe that $g$ is bounded and define $r_{o}=\alpha^{-1}(2\|g\|_{\infty}).$ For all $u\in X$, and all $y\in X$ such that
$d(y,u)>r_{o}$, we have
\[
g(y)+\alpha\bigl(d(u,y)\bigr)>-\|g\|_{\infty}+\alpha(r_{o})=
\|g\|_{\infty}.
\]
Since $Q_{\alpha}g\leq\|g\|_{\infty}$, we conclude that $Q_{\alpha
}g(u)=\inf_{d(y,u)\leq r_{o}}  \{g(y)+\alpha(d(u,y)) \}.$
Let $u_{o}\in X$, and let $B_{o}$ be the closed ball of center $u_{o}$
and radius $2r_{o}$. If $u\in B_{o}$, then $Q_{\alpha}g(u)=\inf_{y\in
B_{o}} \{g(y)+\alpha(d(u,y)) \}.$ Now, if $y\in B_{o}$, we
see that for all $u,v\in B_{o}$,
\begin{eqnarray*}
&&\bigl|\alpha\bigl(d(u,y)\bigr)-\alpha\bigl(d(v,y)\bigr)\bigr|\\
&&\qquad\leq\bigl |d(v,y)-d(u,y)\bigr|
\max_{t\in
[0,1]} \alpha'_+\bigl(td(u,y)+(1-t)d(v,y)\bigr)
\\
&&\qquad\leq L_{o}d(u,v),
\end{eqnarray*}
with $L_{o}=\alpha_+'(4r_{o}).$ The map $B_{o}\to\R\dvtx  u\mapsto
Q_{\alpha}g(u)$ is an infimum of $L_{o}$-Lipschitz functions on
$B_{o}$, so it is $L_{o}$-Lipschitz on $B_{o}$. This ends the proof.
\end{pf}

\begin{lem} \label{lemgradient+}
Let $\alpha$ be a Young function. For any integer $n$, any $t \geq0$
and any bounded continuous function $f \dvtx  X^n \to\mathbb{R}$,
\[
\sum_{i=1}^n \alpha^* \bigl(t \bigl|
\nabla_i^+ Q_{\alpha,n}f\bigr| \bigr) \leq t\xi_\alpha(t)
\bigl(Q_{\alpha,n}f(x)-f\bigl(y^x\bigr)\bigr),
\]
where $y^x\in X^n$ is any point such that $Q_{\alpha,n}f(x)=f(y^x)+
\sum_{j=1}^n \alpha(d(x_j,y^x_j)).$
\end{lem}
\begin{pf}
Fix $n$, $t \geq0$ and a bounded function $f \dvtx  X^n \to\mathbb{R}^+$.
For $x=(x_1,\dots,x_n) \in X^n$, $i\in\{1,\dots,n\}$ and $z \in X$,
we shall use the following notation:
\[
{\bar x}^i z =(x_1,\dots,x_{i-1},z,x_{i+1},
\dots,x_n).
\]
Let $x\in X^n$; since $f$ is bounded continuous and closed balls in $X$
are assumed to be compact, it is not difficult to show that there
exists $y^x \in X^n$ such that
\[
Q_{\alpha,n}f(x)=f\bigl(y^x\bigr)+ \sum
_{j=1}^n \alpha\bigl(d\bigl(x_j,y^x_j
\bigr)\bigr).
\]
For all $z\in X$ and all $1\leq i\leq n$, we have also $Q_{\alpha
,n}f(\bar{x}^{i}z)\leq f(y^x)+\break\sum_{j\neq i}\alpha
(d(x_{j},  y^x_{j}))+\alpha(d(z,y^x_{i})).$ Since the maps $u\mapsto
[u]_{+}$ and $\alpha$ are nondecreasing, it holds
\begin{eqnarray*}
\bigl[Q_{\alpha,n}f\bigl(\bar{x}^iz\bigr)-Q_{\alpha,n}f(x)
\bigr]_{+} &\leq&\bigl[\alpha \bigl(d\bigl(z,y^x_{i}
\bigr)\bigr) - \alpha\bigl(d\bigl(x_i,y_{i}^x
\bigr)\bigr) \bigr]_{+}
\\
&\leq&\bigl[\alpha \bigl(d(z,x_{i}) +d\bigl(x_{i},y^x_{i}
\bigr)\bigr) - \alpha\bigl(d\bigl(x_i,y_{i}^x
\bigr)\bigr) \bigr]_{+}.
\end{eqnarray*}
Therefore,
\begin{eqnarray*}
\bigl|\nabla_i^+ Q_{\alpha,n}f\bigr|(x)&\leq& \limsup_{z\to x_{i}}
\frac{
[\alpha(d(z,x_i)+d(x_i,y^x_i))-\alpha(d(x_i,y^x_i))
]_+}{d(z,x_i)}
\\
&=& \alpha_+'\bigl( d\bigl(x_i,y^x_i
\bigr)\bigr).
\end{eqnarray*}
Hence, by the very definition of $\xi_\alpha$,
\begin{eqnarray*}
\sum_{i=1}^n \alpha^* \bigl(t \bigl|
\nabla_i^+ Q_{\alpha,n}f\bigr| \bigr) &\leq& \sum
_{i=1}^n \alpha^* \bigl(t\alpha_+'
\bigl( d\bigl(x_i,y^x_i\bigr)\bigr) \bigr)
\\
& \leq& t\xi_\alpha(t) \sum_{i=1}^n
\alpha \bigl( d\bigl(x_i,y^x_i\bigr) \bigr)
\\
& =& t\xi_\alpha(t) \bigl( Q_{\alpha,n}f(x) - f
\bigl(y^x\bigr) \bigr) .
\end{eqnarray*}
\upqed\end{pf}

\subsection{Modified log-Sobolev inequality minus}\label{lsiminus}

In this section we prove the second part of Theorem \ref
{OV-thm-intro}, that we restate (in a slightly stronger form) below,
namely that the modified log-Sobolev inequality
minus $\mathbf{LSI}^-_\alpha$ implies
the transport-entropy inequality ${\mathbf{T}_\alpha}$.
Let us define [recall the defintion of $\xi_\alpha$ given in \eqref{xi}]
\[
t_\alpha=\sup \bigl\{ t\in\R^+, \xi_\alpha(t)<+\infty \bigr\}.
\]
Note that, by Lemma~\ref{lemestimationxi}, if $\alpha$ satisfies
the $\Delta_2$-condition, then $t_\alpha\geq1$.
%
\begin{thmm}\label{OV-thm-}
Let $\mu$ be a probability measure on $X$ and $\alpha$ a Young
function satisfying the $\Delta_2$-condition.
If $\mu$ satisfies (\ref{LSIminusA}) for some constant $A>0$,
then~$\mu$ satisfies ${\mathbf{T}_\alpha}(B^-)$ with
$B^-=\lim_{t\rightarrow t_\alpha} \frac1{t} \exp \{\int_{0}^{t} \frac{A\xi_\alpha(u)}{u(1+A\xi_\alpha(u))} \,du \}$.
\end{thmm}

For more comprehension and to complete the proof of part (ii) of
Theorem~\ref{OV-thm-intro}, let us prove that $B^- \leq C^-$. If
$r_\alpha>1$, then by Lemma~\ref{lemestimationxi}, $t_\alpha
=+\infty$. Moreover, using that
$\frac{1}{t}=\exp \{- \int_1^t \frac{1}{u} \,du  \}$, $t
\geq1$, one has
\begin{eqnarray*}
\log B^-&=& \int_0^1 \frac{A\xi_\alpha(u)}{u(1+A\xi_\alpha(u))} \,du
- \int_1^{+\infty}\frac{1}{u(1+A\xi_\alpha(u))} \,du
\\
&\leq& \int_0^1\frac{A(p_\alpha-1)u^{{1}/{(p_\alpha
-1)}}}{u(1+A(p_\alpha-1)u^{{1}/{(p_\alpha-1)}})} \,du \\
&&{}-
\int_1^{+\infty}\frac{1}{u(1+A(p_\alpha-1)u^{{1}/{(r_\alpha-1)}})} \,du
\\
&=& \log C^-,
\end{eqnarray*}
with
\[
C^-= \bigl(1+A(p_\alpha-1)\bigr)^{p_\alpha-r_\alpha}\bigl(A(p_\alpha-1)
\bigr)^{r_\alpha-1}.
\]
When $r_\alpha=1$, since $t_\alpha\geq1$ and using the fact that the function
\[
t\rightarrow\frac1{t} \exp \biggl\{\int_{0}^{t}
\frac{A\xi_\alpha
(u)}{u(1+A\xi_\alpha(u))} \,du \biggr\}
\]
is nonincreasing, we get
\[
B^-\leq\exp \biggl\{\int_{0}^{1}
\frac{A\xi_\alpha(u)}{u(1+A\xi_\alpha(u))} \,du \biggr\}\leq\bigl(1+A(p_\alpha-1)
\bigr)^{p_\alpha-1}=C^-.
\]

\begin{pf*}{Proof of Theorem~\ref{OV-thm-}}
The proof of Theorem~\ref{OV-thm-} follows essentially the lines of
the proof of Theorem~\ref{OV-thm}.
Let $n\in\N^*$; thanks to the tensorisation property
of (\ref{LSIminusA}), it holds
%
\begin{equation}
\label{mlsi-} \ent_{\mu^n}\bigl(e^g\bigr)\leq A\int\sum
_{i=1}^n \alpha^*\bigl(\bigl|
\nabla_i^-g\bigr|\bigr) e^g \,d\mu^n
\end{equation}
for any $g \dvtx  X^n \to\mathbb{R}$ separately locally Lipschitz and
bounded. Take a nonnegative bounded continuous function $f\dvtx  X^n \to
\mathbb{R}^+$. Recall that $P_{\alpha,n}f(x)=\sup_{y\in X^n}  \{
f(y) - \sum_{i=1}^n \alpha(d(x_i,y_i))  \}.$ Since $P_{\alpha
,n} f=-Q_{\alpha,n}(-f)$, it follows from Lemma~\ref{locallyLip}
that $P_{\alpha,n}f$ is separately locally Lipschitz. Applying \eqref
{mlsi-} to $g=tP_{\alpha,n}f$, $t>0$, one gets
\[
\ent_{\mu^n}\bigl(e^{tP_{\alpha,n}f}\bigr) \leq A\int\sum
_{i=1}^n \alpha^* \bigl(t \bigl|\nabla_i^-
P_{\alpha,n} f\bigr| \bigr) e^{tP_{\alpha,n}f} \,d\mu^n.
\]
Observe that $P_{\alpha,n} f=-Q_{\alpha,n}(-f)$ and that $|\nabla^-(-h)|=|\nabla^+h|$, for all $h\dvtx X\to\R$. So applying Lemma \ref
{lemgradient+}, we see that for all $x\in X^n$, there is some $y^x\in
X^n$ such that
\begin{eqnarray*}
\sum_{i=1}^n \alpha^* \bigl(t \bigl|
\nabla_i^- P_{\alpha,n} f\bigl| \bigr) (x)&=&\sum
_{i=1}^n \alpha^* \bigl(t \bigl|\nabla_i^+
Q_{\alpha
,n}(-f)\bigr| \bigr) (x)
\\
&\leq& t\xi_\alpha(t) \bigl(Q_{\alpha,n}(-f) (x)+f
\bigl(y^x\bigr) \bigr)
\\
&\leq& t\xi_\alpha(t) \bigl(\|f\|_\infty- P_{\alpha,n}f(x)
\bigr).
\end{eqnarray*}
So we get the following inequality:
\[
\ent_{\mu^n}\bigl(e^{tP_{\alpha,n}f}\bigr) \leq A t \xi_\alpha(t)
\int\bigl(\|f\|_\infty- P_{\alpha,n}f\bigr) e^{tP_{\alpha,n}f} \,d
\mu^n.
\]
As in the proof of Theorem~\ref{OV-thm}, we proceed with the Herbst argument. Set $H(t)=\int e^{tP_{\alpha,n}f}\,d\mu^n$,
$t \in(0,t_\alpha)$. Since
$\ent_{\mu^n}(e^{tP_{\alpha,n}f}) = tH'(t) - H(t) \log H(t)$, the
latter can be rewritten as
\[
\bigl(t+At\xi_\alpha(t)\bigr) H'(t)\leq H(t) \log H(t) +
A t \xi_\alpha(t)\| f\|_\infty H(t) \qquad\forall t
\in(0,t_\alpha).
\]
Set $W(t)=\frac{1}{t} \log H(t)$, $t\in(0,t_\alpha)$, so that the
previous differential inequality reduces to
\[
W'(t) t\bigl(1+A\xi_\alpha(t)\bigr) \leq-A
\xi_\alpha(t) W (t) + A \xi_\alpha(t)\|f\|_\infty.
\]
Set $c(t)=\exp \{-\int_{0}^t \frac{A\xi_\alpha(u)}{u(1+A\xi_\alpha(u))} \,du \}$ [which belongs to $(0,1)$
thanks to Lem\-ma~\ref{lemestimationxi}].
Since $\lim_{t \to0} W(t)= \mu^n(P_{\alpha,n}f)$, solving the
latter differential inequality, we easily get that for all $t \in
(0,t_\alpha)$,
\[
H(t) \leq e^{ t c(t) \mu^n(P_{\alpha,n}f) } e^{t \|f\|_\infty(1-c(t))} .
\]
Applying Corollary~\ref{cor} yields that ${\mathbf{T}_\alpha}(1/(tc(t)))$
holds for all $t \in(0,t_\alpha)$.
Observing that the function $t\rightarrow tc(t)$ is nondecreasing on
$(0,t_\alpha)$, the proof is completed by optimizing in $t$.
\end{pf*}

\section{\texorpdfstring{From $(\tau)$-log-Sobolev inequality to transport-entropy inequality}
{From (tau)-log-Sobolev inequality to transport-entropy inequality}} \label{GRS-section}

In this section, we prove the second part [$(2)\Rightarrow(1)$] of
Theorem~\ref{GRS-thm-intro}. Observe that ${\mathbf{T}_\alpha}(C/\lambda
)$ is equivalent to ${\mathbf{T}_{\lambda\alpha}}(C)$.
Hence, changing $\alpha$ into $\lambda\alpha$, we can restate the
first part of Theorem~\ref{GRS-thm-intro} as follows.\vadjust{\goodbreak}
%
\begin{thmm}\label{GRS-thm}
Let $\mu$ be a probability measure on $X$ and $\alpha$ a Young
function satisfying the $\Delta_2$-condition. Let $p_\alpha>1$ be
defined by \eqref{defrp}.
If $\mu$ satisfies $(\tau)-\mathbf{LSI}_\alpha(1,A)$ for some $A >0$,
then $\mu$ satisfies ${\mathbf{T}_\alpha}(C)$ with
\[
C=\kappa_{p_\alpha}\max(A,1)^{p_\alpha-1},
\]
where $\kappa_{p_\alpha}=\frac{p_\alpha^{p_\alpha(p_\alpha
-1)}}{(p_\alpha-1)^{(p_\alpha-1)^2}}$.
\end{thmm}

Two proofs are given below. The first one exactly follows the lines of
the proof of $\mathbf{LSI}_\alpha^-\Rightarrow{\mathbf{T}_\alpha}$, whereas
the second one uses the equivalence between transport-entropy
inequalities and dimension-free concentration established in \cite
{gozlan} together with a change of metric argument.

In each proof, the first step is to tensorise the $(\tau)$-log-Sobolev
inequality. Let $n\in\N^*$; using the sub-additivity property \eqref
{subadditivity} of the entropy functional, we see that
$(\tau)-\mathbf{LSI}_\alpha(1,A)$ implies that
%
\begin{equation}
\label{eq1proofGRS} \ent_{\mu^n}\bigl(e^h\bigr)\leq A \int
\sum_{i=1}^n \bigl(h-Q_{\alpha}^{(i)}h
\bigr) e^h \,d\mu^n \qquad\forall h\dvtx X^n\to
\mathbb{R} ,
\end{equation}
where $Q_{\alpha}^{(i)}$ is the inf-convolution operator with respect
to the $i$th coordinate, namely
\[
Q_\alpha^{(i)}h(x)=Q_\alpha(h_{i,x})
(x_i)=\inf_{y \in X} \bigl\{ h\bigl({\bar x}^i y
\bigr) + \alpha\bigl(d(x_i,y)\bigr) \bigr\}
\]
(using the notation introduced in Section~\ref{OV-section}).

As in the proof of Theorem~\ref{OV-thm-}, applying \eqref{eq1proofGRS} to $h=tP_{\alpha,n} g$, $t\geq0$ where $g$ belongs to some class
of functions, we get
%
\begin{equation}
\label{eq2proofGRS} \ent_{\mu^n}\bigl(e^{tP_{\alpha,n}g}\bigr)\leq A \int
\sum_{i=1}^n \bigl(tP_{\alpha,n}g-Q_{\alpha}^{(i)}
(tP_{\alpha,n}g ) \bigr) e^{tP_{\alpha,n}g} \,d\mu^n.
\end{equation}
As a main difference, the class of functions $g$ differs in each proof.
In the first one, $g$ is any nonnegative bounded separately locally
Lipschitz function, whereas in the second proof, $g$
is globally Lipschitz in some sense.

For both proofs, the next step is to bound efficiently the right-hand
side of \eqref{eq2proofGRS}, in order to use some Herbst argument.
This bound will be given by the following lemma.
%
\begin{lem} \label{petitlemmebis}
Let $\alpha$ be a Young function satisfying the $\Delta_2$-condition,
and let $p_\alpha>1$ be defined by \eqref{defrp}. For any bounded
continuous function $g \dvtx  X^n \to\mathbb{R}$, for any $x\in X^n$ and
$t\in[0,1)$,
\[
\sum_{i=1}^n \bigl(tP_{\alpha,n}g(x)
- Q_\alpha^{(i)}( tP_{\alpha
,n}g) (x) \bigr) \leq t
\varepsilon(t) \Biggl( \sum_{i=1}^n \alpha
\bigl(d\bigl(x_i,y_i^x\bigr)\bigr) \Biggr),
\]
where $y^x\in X^n$ is any point such that $P_{\alpha
,n}g(x)=g(y^x)-\sum_{i=1}^n \alpha(d(x_i,y_i^x))$, and where
\[
\varepsilon(t)=\frac{1}{ (1-t^{{1}/{(p_\alpha-1)}}
)^{p_\alpha-1}}-1 \qquad \forall t\in[0,1).
\]
\end{lem}
We postpone the proof of Lemma~\ref{petitlemmebis} to the end of the section.

\subsection{A first proof}
The first proof of Theorem~\ref{GRS-thm}
mimics the one of the implication $\mathbf{LSI}_\alpha^-\Rightarrow{\mathbf{T}_\alpha}$.
\begin{pf*}{Proof of Theorem~\ref{GRS-thm}}
Using \eqref{eq2proofGRS}, Lemma~\ref{petitlemmebis} ensures that
for every nonnegative locally Lipschitz bounded function $g$, for
every $t\in[0,1)$,
\[
\bigl(t+At\varepsilon(t)\bigr)H'(t) \leq H(t)\log H(t) + A t
\varepsilon(t) \| g\|_\infty H(t) ,
\]
where $H(t)=\int e^{tP_{\alpha,n}g} \,d\mu^n$.

Solving this differential inequality, exactly as in the proof of
Theorem~\ref{OV-thm-} (we omit details), leads to
\[
\int e^{P_{\alpha,n}g} \,d\mu^n \leq e^{c \mu^n(P_{\alpha,n}g)}
e^{\|
g\|_\infty(1-c)},
\]
with $c=1/C$,
\[
C=\exp\int_0^1 \frac{A \varepsilon(t)}{t(1+ A \varepsilon(t))} \,dt .
\]
The inequality ${\mathbf{T}_\alpha}(C)$ then follows from Corollary~\ref{cor}.

Now, let us estimate the constant $C$. By convexity, one has for every
$v\in[0,1]$,
\[
(1-v)-(1-v)^{p_\alpha}\leq(p_\alpha-1)v.
\]
This inequality easily implies that for all $t\in[0,1)$,
$\frac{ \varepsilon( t)}{t}\leq(p_\alpha-1) \varepsilon'(t).$
Consequently, we obtain for all $u\in[0,1)$,
\begin{eqnarray*}
\log C&\leq& (p_\alpha-1)\int_0^u
\frac{A \varepsilon'(t)}{1+ A
\varepsilon(t)} \,dt+ \int_u^1
\frac{A \varepsilon(t)}{t(1+ A
\varepsilon(t))} \,dt
\\
&\leq& (p_\alpha-1)\log\bigl(1+ A \varepsilon(u)\bigr)-\log u.
\end{eqnarray*}
Optimizing in $u$, we get
\begin{eqnarray*}
C & \leq &\inf_{u\in(0,1)} \frac{(1+ A \varepsilon(u))^{p_\alpha-1}}u\leq \inf_{u\in(0,1)}
\frac{(1+ \varepsilon(u))^{p_\alpha-1}}u\max (A,1)^{p_\alpha-1}
\\
& =& \kappa_{p_\alpha}\max(A,1)^{p_\alpha-1},
\end{eqnarray*}
with $\kappa_{p_\alpha}=\frac{p_\alpha^{p_\alpha(p_\alpha
-1)}}{(p_\alpha-1)^{(p_\alpha-1)^2}}$.
\end{pf*}

\subsection{A second proof} The idea of this second proof is to prove
the theorem in the particular case of the functions $\alpha_p(x)=x^p$
and then to treat the general case by a change of metric argument.
\subsubsection{${\mathbf{T}_p}$ inequalities}
Let us introduce some notation and definitions. When $\alpha(x)=\alpha_p(x)=|x|^p$, we will use the notation ${\mathbf{T}_p}(C)$ and
$(\tau )-\mathbf{LSI}_p$ instead of ${\mathbf{T}_{\alpha_p}}(C)$ and
$(\tau )-\mathbf{LSI}_{\alpha_p}$. Let $n\in\N^*$; a function $f\dvtx X^n\to\R$ is
said to be $(L,p)$-Lipschitz $L>0,p>1$, if
\[
\bigl|f(x)-f(y)\bigr|\leq L \Biggl(\sum_{i=1}^n
d^p(x_i,y_i) \Biggr)^{1/p} \qquad\forall
x,y\in X^n.
\]

We recall the following result from~\cite{gozlan}.
%
\begin{thmm}\label{gozlanp}
The probability $\mu$ verifies the transport-entropy inequality ${\mathbf{T}_p}(C)$, for some $C>0$ if and only if it enjoys the following
dimension free concentration property:
for all $n\in\N^*$ and all $f\dvtx X^n\to\R$ such that
\[
\bigl|f(x)-f(y)\bigr|\leq L \Biggl(\sum_{i=1}^n
d^p(x_i,y_i) \Biggr)^{1/p} \qquad\forall
x,y\in X^n
\]
for some $L>0$, it holds
\[
\mu^n\bigl(f\geq\mu^n(f)+u\bigr)\leq\exp
\bigl(-u^{p}/\bigl(L^pC\bigr)\bigr) \qquad\forall u\geq0.
\]
\end{thmm}
So to show that a ${\mathbf{T}_p}$ inequality holds, it is enough to prove
the right concentration inequality.

We will use the following result to estimate the right-hand side of
\eqref{eq2proofGRS}.
%
\begin{lem}\label{omegap}
Let $p>1$; there exists a constant $\omega_p\geq1$ such that for all
$n\in\N^*$ and all $(L,p)$-Lipschitz function $f\dvtx X^n\to\R$, and all
$x\in X^n$, the function
\[
X^n\to\R\dvtx y\mapsto f(y)-\sum_{i=1}^n
d^p(x_i,y_i)
\]
attains its maximum on the closed ball
\[
\Biggl\{y\in X^n ; \sum_{i=1}^n
d^p(x_i,y_i)\leq \biggl(\frac
{L}{\omega_p}
\biggr)^{q} \Biggr\},\qquad \mbox{with } q=\frac{p}{p-1}.
\]
When $(X,d)$ is geodesic (see below), then one can take $\omega_p=p.$
\end{lem}
Recall that $(X,d)$ is geodesic, if for all $x,y\in X$, there is a path
$(z_{t})_{t\in[0,1]}$ joining $x$ to $y$ and such that
$d(z_s,z_t)=|s-t|d(x,y)$, for all $s,t\in[0,1].$ This notion
encompasses the case of Riemannian manifolds.

The proof of the lemma is at the end of the section.
%
\begin{thmm}\label{GRS-p}
Let $p\geq2$; if $\mu$ verifies the $(\tau) - \mathbf{LSI}_{p}(1,A)$,
then it verifies ${\mathbf{T}_p}(C)$, with $C= (a_p \max(1 ; A)
)^{p-1},$ with $a_p=\inf_{t\in(0,1)} \{\frac{1}{t^{q-1}}
(1 +\frac{(p/\omega_p)^q}{p-1}\int_0^t \frac{\varepsilon(u)}{u}
\,du ) \}.$
\end{thmm}
%
\begin{rem} Let us compare the constants appearing in Theorems \ref
{GRS-thm} and~\ref{GRS-p} for $p=2$.
When $p=2$, Theorem~\ref{GRS-thm} gives $C_1=4\max(1;A)$, and Theorem
\ref{GRS-p} gives $C_2=a_2 \max(1;A)$. A simple calculation shows
that when $p=2$, $a_2=\inf_{s\in(0,1)} \{\frac{1-(2/\omega_2)^2\ln(1-s)}{s} \}.$ If $\omega_2=1$, then $a_2\simeq7,5$,
and $C_1$ is smaller than~$C_2$. But if $\omega_2=2$ [which is the
case, when $(X,d)$ is geodesic], then $a_2\simeq3,14,$ and $C_2$ is
smaller than $C_1.$
\end{rem}

\begin{pf*}{Proof of Theorem~\ref{GRS-p}}
Take a $(L,p)$-Lipschitz function $g\dvtx X^n\to\R.$
To bound the right-hand side of \eqref{eq2proofGRS}, we use Lemmas
\ref{petitlemmebis} and~\ref{omegap}.
\[
\sum_{i=1}^n \bigl(tP_{\alpha_p,n}g(x)
- Q_{\alpha_p}^{(i)}( tP_{\alpha_p,n}g) (x) \bigr) \leq t
\varepsilon(t) (L/\omega_p )^q.
\]
So, letting $H(t)=\int e^{tP_{\alpha_p,n}g} \,d\mu^n,$ \eqref{eq2proofGRS} provides
\[
tH'(t)-H(t)\log H(t)\leq At\varepsilon(t)H(t) (L/
\omega_p)^{q} \qquad\forall t\in[0,1).
\]
Equivalently, the function $K(t)=\frac{1}{t}\log H(t)$ verifies
$K'(t)\leq A(L/\omega_p)^{q}\frac{\varepsilon(t)}{t}.$ Since
$K(t)\to\mu^n(P_{\alpha_p,n}g)$ when $t\to0^+$, we conclude that
%
\begin{equation}
\label{eqproofGRS-p} \qquad\int e^{tP_{\alpha_p,n}g} \,d\mu^n\leq\exp \bigl( t
\mu^n(P_{\alpha
_p,n}g) + tA(L/\omega_p)^{q}k(t)
\bigr)\qquad \forall t\in[0,1),
\end{equation}
where $k(t)=\int_0^t \frac{\varepsilon(u)}{u} \,du,\ t\in[0,1).$
Since $g$ is $(L,p)$-Lipschitz, it holds
\begin{eqnarray*}
0\leq P_{\alpha_p,n}g(x)-g(x)&\leq&\sup_{y\in X^n} \Biggl\{ L \Biggl(\sum
_{i=1}^n d^p(x_i,y_i)
\Biggr)^{1/p}-\sum_{i=1}^n
d^p(x_i,y_i) \Biggr\}
\\
&=&\sup_{r\geq0}\bigl\{Lr-r^{p}\bigr\}=(p-1) (L/p
)^{q}.
\end{eqnarray*}
Plugging the inequalities $g\leq P_{\alpha_p,n}g$ and $\mu^n
(P_{\alpha_p,n}g )\leq\mu^n(g)+(p-1)  (L/p )^{q}$
into \eqref{eqproofGRS-p}, we get
\[
\int e^{tg} \,d\mu^n\leq\exp \bigl( t\mu^n(g)+t(p-1)
(L/p )^{q} + tA (L/\omega_p )^{q}k(t) \bigr)\qquad
\forall t\in[0,1).
\]
Applying this inequality to $g=f/t$ with $f$ a $(L,p)$-Lipschitz
function, we get
\begin{eqnarray*}
\int e^{f-\mu^n(f)} \,d\mu^n\leq\exp \biggl(\frac
{(L/p)^{q}}{t^{q-1}}
\bigl(p-1 +A (p/\omega_p)^q k(t) \bigr) \biggr)\qquad \forall t
\in(0,1).
\end{eqnarray*}
So, optimizing over $t\in(0,1)$ yields
%
\begin{equation}
\label{eq2altproof} \int e^{f-\mu^n(f)} \,d\mu^n\leq\exp
\bigl((L/p)^{q}(p-1) \max(1;A) a_{p} \bigr),
\end{equation}
with
\[
a_{p}=\inf_{t\in(0,1)} \biggl\{\frac{1}{t^{q-1}} \biggl(1 +
\frac
{(p/\omega_p)^q}{p-1}\int_0^t \frac{\varepsilon(u)}{u}
\,du \biggr) \biggr\}.
\]
Using Chebyshev's argument, we derive from \eqref{eq2altproof} that
\[
\mu^n\bigl(f\geq\mu^n(f)+u\bigr)\leq\exp
\bigl(-u^{p}/\bigl(L^{p}C\bigr)\bigr)\qquad \forall u\geq0,
\]
with
$C=  (a_p\max(A;1) )^{p-1}.$
Applying Theorem~\ref{gozlanp}, we conclude that $\mu$ verifies
${\mathbf{T}_p}(C)$.
\end{pf*}

\subsubsection{Extension via a change of metric}
A change of metric technique, which is explained in the lemma below,
enables us to reduce the study of the transport-entropy inequalities
${\mathbf{T}_\alpha}$ to the study of the inequalities ${\mathbf{T}_p}$, $p>1.$
%
\begin{lem}\label{lemchangemetric}
Let $\alpha$ be a Young function satisfying the $\Delta_2$-condition,
and let $p_\alpha>1$ be defined by \eqref{defrp}. The function
$x\mapsto\alpha(x)^{1/p_\alpha}$ is subadditive on~$\R^+$:
\[
\alpha^{1/p_\alpha}(x+y)\leq\alpha^{1/p_\alpha}(x)+\alpha^{1/p_\alpha}(y)
\qquad\forall x,y\in\R^+.
\]
As a consequence, $d_\alpha(x,y)=\alpha^{1/p_\alpha}(d(x,y))$,
$x,y\in X$ is a distance on $X$.
\end{lem}
The proof of Lemma~\ref{lemchangemetric} is at the end of the section.

\begin{pf*}{Proof of Theorem~\ref{GRS-thm}}
Let $\alpha$ be a Young function and $\mu$ a probability on $X$.
According to Lemma~\ref{lemchangemetric}, the function $d_\alpha
(x,y)=\alpha^{1/p_\alpha} (d(x,y))$, $x,y\in X$ is a metric on $X.$
Furthermore, it is clear that $\mu$ verifies $(\tau) -
\mathbf{LSI}_\alpha(1,A)$ [resp., ${\mathbf{T}_\alpha}(C)$] on $(X,d)$ if and only
if $\mu$ verifies $(\tau) - \mathbf{LSI}_{p_\alpha}(1,A)$ [resp., ${\mathbf{T}_{p_\alpha}}(C)$]
on $(X,d_\alpha)$. We immediately deduce from
Theorem~\ref{GRS-p} that if $\mu$ verifies $(\tau) - \mathbf{LSI}_\alpha
(1,A)$, then it verifies ${\mathbf{T}_{\alpha}}(C)$, with $C=(a_{p_\alpha
}\max(1 ;\break A))^{p_\alpha-1}$, where
\[
a_{p_\alpha}=\inf_{t\in(0,1)} \biggl\{\frac{1}{t^{q_\alpha-1}} \biggl(1 +
\frac{p_\alpha^{q_\alpha}}{p_\alpha-1}\int_0^t \frac
{\varepsilon(u)}{u}
\,du \biggr) \biggr\}
\]
(since $\omega_{p_\alpha}\geq1$) and $\varepsilon$ defined in Lemma
\ref{petitlemmebis}. Note that the constant $C$ obtained using this
approach is in general bigger than the constant obtained in the first
proof of Theorem~\ref{GRS-thm}.
\end{pf*}

\subsection{Proofs of the technical lemmas}
\mbox{}
\begin{pf*}{Proof of Lemma~\ref{petitlemmebis}}
Fix $t >0$, $x \in X^n$ and $i\in\{1,\ldots,n\}$.
Then
\[
tP_{\alpha,n}g(x) - Q_\alpha^{(i)}(tP_{\alpha,n}g)
(x) = \sup_{z \in
X} \bigl\{ \bigl(tP_{\alpha,n}g(x) -
tP_{\alpha,n}g\bigl({\bar x}^iz\bigr)\bigr) - \alpha
\bigl(d(x_i,z)\bigr) \bigr\}.\vadjust{\goodbreak}
\]
Let $y^x$ be such that $P_{\alpha,n}g(x) = g(y^x) - c(x,y^x)$, where
$c(x,y)=\sum_{i=1}^n \alpha(d(x_i,\break y_i)).$ By choosing $w=y^x$ in the
expression below, it holds
\begin{eqnarray*}
P_{\alpha,n}g(x) - P_{\alpha,n}g\bigl({\bar x}^iz\bigr) &
=& g\bigl(y^x\bigr) - c\bigl(x,y^x\bigr) -
\sup_{w \in X^n} \bigl\{ g(w) - c\bigl({\bar x}^iz,w\bigr) \bigr\}
\\
& \leq& c\bigl({\bar x}^iz,y^x\bigr) - c
\bigl(x,y^x\bigr)
\\
& = &\alpha\bigl(d\bigl(z,y^x_i\bigr)\bigr) - \alpha
\bigl(d\bigl(x_i,y^x_i\bigr)\bigr)
\\
&=&d_\alpha^{p_\alpha}\bigl(z,y^x_i
\bigr)-d_\alpha^{p_\alpha}\bigl(x_i,y^x_i
\bigr)
\\
& \leq&\bigl(d_\alpha\bigl(x_i,y^x_i
\bigr) + d_\alpha(x_i,z)\bigr)^{p_\alpha} -
d_\alpha^{p_\alpha}\bigl(x_i,y^x_i
\bigr) ,
\end{eqnarray*}
where, in the last line, we used the triangular inequality for the
distance $d_\alpha$ defined in Lemma~\ref{lemchangemetric}. Hence,
optimizing yields
\begin{eqnarray*}
&& tP_{\alpha,n}g(x) - Q_\alpha^{(i)}( tP_{\alpha,n}g)
(x)
\\
& &\qquad\leq \sup_{z \in X} \bigl\{ t \bigl[ \bigl(d_\alpha
\bigl(x_i,y^x_i\bigr) + d_\alpha
(x_i,z)\bigr)^{p_\alpha} - d_\alpha^{p_\alpha}
\bigl(x_i,y^x_i\bigr) \bigr] -
d_\alpha^{p_\alpha}(x_i,z) \bigr\}
\\
&&\qquad = \sup_{r >0} \bigl\{ t \bigl[ \bigl(d_\alpha
\bigl(x_i,y^x_i\bigr) + r\bigr)^{p_\alpha}
- d^{p_\alpha}_\alpha\bigl(x_i,y^x_i
\bigr) \bigr] - r^{p_\alpha} \bigr\}
\\
&&\qquad=t\varepsilon(t)\,d_\alpha^{p_\alpha}\bigl(x_i,y_i^x
\bigr)= t\varepsilon (t)\alpha\bigl(d\bigl(x_i,y_i^x
\bigr)\bigr).
\end{eqnarray*}
Taking the sum, we get the result.
\end{pf*}

\begin{pf*}{Proof of Lemma~\ref{omegap}}
Let $y^x$ be a point where the function $y\mapsto f(y)-\sum_{i=1}^n
d^p(x_i,y)$ reaches its maximum. Then, for all $z\in X^n$, it holds
\begin{eqnarray*}
\sum_{i=1}^n d^p
\bigl(x_i,y^x_i\bigr) &\leq& f
\bigl(y^x\bigr)-f(z)+ \sum_{i=1}^n
d^p(x_i,z_i)
\\
&\leq & L \Biggl(\sum_{i=1}^n
d^p\bigl(z_i,y^x_i\bigr)
\Biggr)^{1/p}+\sum_{i=1}^n
d^p(x_i,z_i).
\end{eqnarray*}
Choosing $z=x$, we get $\sum_{i=1}^n d^p(x_i,z_i) \leq L^q.$

Now, assume that $(X,d)$ is geodesic. Then the product space
$(X^n,d^{(n)})$ with $d^{(n)}(x,y)= (\sum_{i=1}^n
d^p(x_i,y_i) )^{1/p}$ is geodesic too. In the calculation above,
take for $z$ a $t$-midpoint of $x$ and $y^x$; that is, choose $z\in
X^n$ such that $d^{(n)}(x,z)=td^{(n)}(x,y^x)$ and
$d^{(n)}(z,y^x)=(1-t)\,d^{(n)}(x,y^x)$, with $t\in[0,1]$. Then, letting
$\ell= d^{(n)}(x,y^x)$, it holds
$\ell^p\leq L (1-t)\ell+ t^p\ell^p,$ and so $\frac{1-t^p}{1-t}\ell^{p-1}\leq L.$ Letting $t\to1$ gives the result.
\end{pf*}

\begin{pf*}{Proof of Lemma~\ref{lemchangemetric}}
Let $\varphi(x)=\alpha^{1/p_\alpha}(x)/x$, $x>0$. Then, by
definition of $p_\alpha$,
\[
\varphi'_+(x)=\frac{\alpha^{1/p_\alpha}(x) ({x\alpha'_+(x)}/{(p_\alpha\alpha(x))}-1 )}{x^2}\leq0.
\]
So $\varphi$ is nonincreasing on $(0,+\infty).$ Thus, if $x>y$,
\begin{eqnarray*}
\alpha^{1/p_\alpha}(x+y)&=&(x+y)\varphi\bigl(x(1+y/x)\bigr)\leq(x+y)\varphi (x)
\\
&=&\alpha^{1/p_\alpha}(x)+y\frac{\alpha^{1/p_\alpha}(x)}{x}\leq \alpha^{1/p_\alpha}(x)+
\alpha^{1/p_\alpha}(y).
\end{eqnarray*}
\upqed\end{pf*}

\section{\texorpdfstring{Holley--Stroock perturbation lemma: Proof of Theorem \protect\ref{holley}}
{Holley--Stroock perturbation lemma: proof of Theorem 1.9}}\label{hs}

In this section, we prove Theorem~\ref{holley}.

\begin{pf*}{Proof of Theorem~\ref{holley}}
The proof follows the line of the original proof~\cite{HS87}; see also
\cite{royer}.
Using the following representation of the entropy,
\[
\ent_\mu (g ) = \inf_{t >0} \biggl\{ \int \biggl(g \log
\biggl( \frac{g}{t} \biggr) - g + t \biggr) \,d\mu \biggr\}
\]
with $g=e^f$, we see that [since $g \log ( \frac{g}{t}  )
- g + t \geq0$]
\[
\ent_{\tilde{\mu}} (g ) \leq\frac{e^{\sup\varphi
}}{Z} \ent_{\mu} (g ) .
\]
Since $\mu$ verifies ${\mathbf{T}_\alpha}(C)$, Theorem~\ref{grs} implies
that, for all $\lambda\in(0,1/C)$,
\begin{eqnarray*}
\ent_{\tilde{\mu}} \bigl(e^f \bigr) \leq \frac{e^{\sup\varphi}}{Z}
\frac{1}{1-\lambda C} \int\bigl(f - Q^\lambda_\alpha f\bigr)
e^f \,d\mu \leq \frac{e^{{\mathrm{Osc}}(\varphi)}}{1-\lambda C} \int\bigl(f - Q^\lambda_\alpha
f\bigr) e^f \,d\tilde{\mu} .
\end{eqnarray*}
In other words, $\tilde\mu$ satisfies $(\tau)-\mathbf{LSI}_\alpha
(\lambda,\frac{e^{{\mathrm{Osc}}(\varphi)}}{1-\lambda C})$, for
any $\lambda\in(0,1/C)$. Now, applying Theorem~\ref{GRS-thm-intro},
we conclude that $\tilde{\mu}$ verifies ${\mathbf{T}_\alpha}
(\widetilde{C} )$, with
\begin{eqnarray*}
\widetilde{C}&=&\kappa_{p_\alpha}\inf_{\lambda\in(0,1/C)} \biggl\{
\frac{1}{\lambda}\frac{1}{(1-\lambda C)^{p_\alpha-1}} \biggr\} e^{(p_\alpha-1)\mathrm{Osc}(\varphi)}
\\
&=& \kappa_{p_\alpha}\frac{p_\alpha^{p_\alpha}}{(p_\alpha
-1)^{p_\alpha-1}}Ce^{(p_\alpha-1)\mathrm{Osc}(\varphi)}= \widetilde {
\kappa}_{p_\alpha}Ce^{(p_\alpha-1)\mathrm{Osc}(\varphi)}.
\end{eqnarray*}
\upqed\end{pf*}

\begin{appendix}\label{app}

\section*{Appendix: Technical results}

In this appendix we prove the technical lemmas on Young functions we
used during the paper.
First, let us prove Lemma~\ref{lemrp}, that we restate below.

\begin{lem}
If $\alpha$ is a Young function satisfying the $\Delta_2$-condition, then
%
\begin{equation}
\label{defrpappendix} r_\alpha:=\inf_{x>0}\frac{x\alpha'_-(x)}{\alpha(x)}
\geq1 \quad\mbox{and}\quad 1<p_\alpha:=\sup_{x>0}\frac{x\alpha'_+(x)}{\alpha
(x)}<+
\infty,
\end{equation}
where $\alpha'_+$ (resp., $\alpha'_-$) denotes the right (resp., left)
derivative of $\alpha.$
\end{lem}

\begin{pf}
Using the convexity of $\alpha$, we see that $\alpha(x)/x\leq\alpha'_-(x)$. This shows that $r_\alpha\geq1.$
On the other hand, the function $\alpha$ is convex, so $\alpha
(2x)\geq\alpha(x)+x\alpha_+'(x)$, for all $x>0$. Since $\alpha$
verifies the $\Delta_2$-condition, there is some constant $K\geq2$\vadjust{\goodbreak}
such that $\alpha(2x)\leq K\alpha(x)$. So we get $x\alpha_+'(x)\leq
(K-1)\alpha(x)$, for all $x>0$. This proves that $p_\alpha<+\infty.$
Let us show that $p_\alpha>1$. Otherwise we would have $r_\alpha
=p_\alpha$ (since $\alpha'_-\leq\alpha'_+$) and so $x\alpha'_-(x)/\alpha(x)=x\alpha'_+(x)/\alpha(x)=1$ for all $x>0.$ This
would imply that $\alpha$ is linear on $[0,\infty)$. This cannot
happen, since by assumption Young functions are increasing and such
that $\alpha'(0)=0$. So $p_\alpha>1.$
\end{pf}
Now let us prove Lemmas~\ref{lemestimationxi} and~\ref{exactcalculation} whose statements are summarized below. Recall that the
function $\xi_\alpha$ is defined by
\[
\xi_\alpha(x):= \sup_{u >0} \frac{\alpha^*(x\alpha_+'(u))}{x\alpha(u)} ,\qquad x > 0.
\]

\begin{lem}
\begin{itemize}
\item Let $\alpha$ be a Young function satisfying the $\Delta_2$-condition, and let $1\leq r_\alpha\leq p_\alpha$, $p_\alpha>1$
be the numbers defined by \eqref{defrpappendix}. Then, it holds
%
\begin{equation}
\label{eqestimationxiappendix} \xi_\alpha(x)\leq(p_\alpha-1)
\max \bigl(x^{{1}/{(p_\alpha
-1)}}; x^{{1}/{(r_\alpha-1)}} \bigr)\qquad \forall x>0,
\end{equation}
with the convention $t^\infty=0$ if $t\leq1$ and $\infty$
otherwise.

\item Let $p_1\geq2$ and $p_2>1$ and let $\alpha=\alpha_{p_1,p_2}$; then $p_\alpha=\max(p_1,p_2)$, and it holds
\[
\xi_\alpha(x)=(p_\alpha-1)x^{{1}/{(p_\alpha-1)}} \qquad\forall x\leq1.
\]
Moreover, for $x\geq1$, it holds
\[
\xi_\alpha(x)=\cases{
\displaystyle p_1
\biggl(\frac{1}{q_2}x^{{1}/{(p_2-1)}}+ \biggl(\frac{1}{q_1}-
\frac{1}{q_2} \biggr)\frac{1}{x} \biggr), &\quad  $\mbox{if } p_1
\geq p_2,$
\vspace*{2pt}\cr
\displaystyle \max \bigl((p_1-1)x^{{1}/{(p_1-1)}};(p_2-1)x^{
{1}/{(p_2-1)}}
\bigr), & \quad $\mbox{if } p_1\leq p_2,$}
\]
where $q_1=p_1/(p_1-1)$ and $q_2=p_2/(p_2-1).$
\end{itemize}
\end{lem}

\begin{pf}
Defining $\omega(x)=\sup_{u>0}\frac{\alpha^*(ux)}{\alpha^*(u)}$,
for all $x\geq0$, we get
\[
\xi_\alpha(x)\leq\frac{\omega(x)}{x} \sup_{u>0}
\frac{\alpha^*(\alpha'_+(u))}{\alpha(u)} \qquad\forall x>0.
\]
From the convexity inequality $\alpha(x)\geq\alpha(u)+(x-u)\alpha'_+(u)$, $x,u\geq0,$ we deduce immediately that $\alpha^*(\alpha'_+(u))=u\alpha'_+(u)-\alpha(u),$ for all $u\geq0.$ Thus
\[
\sup_{u>0}\frac{\alpha^*(\alpha'_+(u))}{\alpha(u)}=p_\alpha-1.
\]
So, all we have to show is that $\omega(x)\leq\max (x^{
{p_\alpha}/{(p_\alpha-1)}} ; x^{{r_\alpha}/{(r_\alpha-1)}}
)$, for all $x\geq0$.

Define $\varphi(u)=\alpha(u)/u^{p_\alpha}$ and $\psi(u)=\alpha
(u)/u^{r_\alpha},$ for all $u>0.$ As in the proof of Lemma~\ref{lemchangemetric},\vadjust{\goodbreak} a simple calculation shows that $\varphi$ is
nonincreasing, and $\psi$ is nondecreasing.
As a result,
\begin{eqnarray*}
\alpha(tu)&\leq& t^{p_\alpha}\alpha(u) \qquad\forall u\geq0,\ \forall t\geq1,
\\
\alpha(tu)&\leq& t^{r_\alpha}\alpha(u) \qquad\forall u\geq0,\ \forall t\in[0, 1].
\end{eqnarray*}
Taking the Fenchel--Legendre transform yields
\begin{eqnarray*}
\alpha^*(v/t)&\geq& t^{p_\alpha}\alpha^*\bigl(v/t^{p_\alpha}\bigr)\qquad
\forall v\geq0,\ \forall t\geq1,
\\
\alpha^*(v/t)&\geq& t^{r_\alpha}\alpha^*\bigl(v/t^{r_\alpha}\bigr)\qquad
\forall v\geq0,\ \forall t\in[0, 1].
\end{eqnarray*}
Equivalently,
\begin{eqnarray*}
\alpha^*(ux)&\leq& x^{{p_\alpha}/{(p_\alpha-1)}}\alpha^*(u)\qquad \forall u\geq0,\ \forall x
\in[0,1],
\\
\alpha^*(ux)&\leq& x^{{r_\alpha}/{(r_\alpha-1)}}\alpha^*(u) \qquad\forall u\geq0,\ \forall x\geq1.
\end{eqnarray*}
And since $r_\alpha\leq p_\alpha$, we conclude that $\omega(x)\leq
\max (x^{{p_\alpha}/{(p_\alpha-1)}} ; x^{{r_\alpha
}/{(r_\alpha-1)}}  ),$ \mbox{$x\geq0.$}

Now, let us calculate $\xi_{\alpha_{p_1,p_2}}$, for $p_1\geq2$,
$p_2>1$. First observe that $\xi_{\lambda\alpha}=\xi_\alpha$ for
all $\lambda>0.$ It will be more convenient to do the calculation with
the function $\alpha:=\bar{\alpha}_{p_1,p_2}=\frac{1}{p_1}\alpha_{p_1,p_2}.$ Let us denote by $q_1=\frac{p_1}{p_1-1}$, $q_2=\frac
{p_2}{p_2-1},$ the conjugate exponents of $p_1$ and $p_2.$ Then the
following identity holds: $\alpha^*=\bar{\alpha}_{q_1,q_2}.$ Let us
show that
\[
\xi_\alpha(x)=(p_\alpha-1)x^{1/(p_\alpha-1)}
\]
for $x\leq1.$ The case $x>1$ is similar and left to the reader. Define
\[
\varphi(u)=\frac{\alpha^*(xu)}{\alpha\circ\alpha'^{ -
1}(u)},\qquad u>0.
\]
We have to distinguish three cases:
\[
\xi_{\alpha}(x)=\frac{1}{x}\max \Bigl(\sup_{u\leq1}
\varphi(u); \sup_{1\leq u\leq1/x}\varphi(u) ; \sup_{u\geq1/x}\varphi(u)
\Bigr).
\]

\textit{Case} 1. $0<u\leq1$. Then $\varphi(u)=(p_1-1)x^{q_1}.$

\textit{Case} 2. $1\leq u\leq1/x$. Then
\[
\varphi(u)=\frac{x^{q_1}}{q_1} \frac{u^{q_1}}{
{u^{q_2}}/{p_2}+{1}/{p_1}-{1}/{p_2}}.
\]
If $p_1\geq p_2$, then the function $\varphi$ is nonincreasing on
$[1,1/x]$, and so
\[
\sup_{1\leq u\leq1/x}\varphi(u)=\varphi(1)=(p_1-1)x^{q_1.}
\]
If $p_1\leq p_2$, then the function $\varphi$ is nondecreasing on
$[1,1/x]$, and so
\[
\sup_{1\leq u\leq1/x}\varphi(u)=\varphi(1/x).
\]

\textit{Case} 3. $u\geq1/x.$ Then
\[
\varphi(u)=\frac{{(xu)^{q_2}}/{q_2}+{1}/{q_1} -
{1}/{q_2}}{{u^{q_2}}/{p_2}+{1}/{p_1} -{1}/{p_2}}.
\]
If $p_1\geq p_2$, the function $\varphi$ is nonincreasing on $[1/x,
\infty)$, and so
\[
\sup_{u\geq1/x}\varphi(u)=\varphi(1/x).
\]
If $p_1\leq p_2$, the function $\varphi$ is nondecreasing on
$[1/x,\infty)$, and so
\[
\sup_{u\geq1/x}\varphi(u)=\lim_{u\to\infty}\varphi(u)=
(p_2-1)x^{q_2}.
\]
Observe, in particular, that $\varphi$ never reaches its supremum at $u=1/x.$
We conclude that
\[
\sup_{u>0} \varphi(u)=\max\bigl( (p_1-1)x^{q_1};
(p_2-1)x^{q_2}\bigr),
\]
and so
\begin{eqnarray*}
\xi_\alpha(x) & =& \max \bigl( (p_1-1)x^{1/(p_1-1)};
(p_2-1)x^{1/(p_2-1)} \bigr)
\\
& =&\bigl(\max(p_1;p_2)-1\bigr)x^{1/(\max(p_1;p_2)-1)}.
\end{eqnarray*}
Since $p_\alpha=\max(p_1;p_2)$, the proof is complete.
\end{pf}
\end{appendix}

\section*{Acknowledgments}
We thank the referees for a careful reading of our paper and for providing
useful feedback.

%


\printaddresses

\end{document}